\documentclass[reqno,a4paper,12pt]{amsart}

\usepackage{mystyle}
\usepackage{etoolbox}

\usepackage[utf8]{inputenc}
\usepackage[T1]{fontenc}
\usepackage{lmodern}
\usepackage[british]{babel}
\usepackage[nopatch=eqnum]{microtype}
\usepackage[headings]{fullpage}
\usepackage[skip=0.5\baselineskip plus 2.75pt]{parskip}

\usepackage[pdfencoding=auto,pdfusetitle]{hyperref}
\hypersetup{colorlinks=true, linkcolor=blue!30!black, citecolor=green!30!black, urlcolor=red!45!black}
\usepackage{breakurl} %%

\usepackage{microtype}

\usepackage{amsfonts,amssymb,bbm,stmaryrd}
\usepackage[scr=boondoxo,scrscaled=1.05]{mathalfa}
\usepackage{amsmath}

\usepackage{mathtools}
\usepackage[noabbrev,capitalise]{cleveref}
\crefformat{equation}{(#2#1#3)}
\crefformat{enumi}{#2#1#3}
\crefformat{enumii}{#2#1#3}

\usepackage{centernot,float,subcaption,url,array,booktabs,colortbl}

\usepackage[shortlabels]{enumitem}

\usepackage{tikz}
\usetikzlibrary{matrix,positioning}
\usetikzlibrary{calc}
\usetikzlibrary{arrows}
\tikzset{> =stealth}

\usepackage{tikz-cd}

\usetikzlibrary{decorations.pathmorphing}

\tikzset{curve/.style={settings={#1},to path={(\tikztostart)
    .. controls ($(\tikztostart)!\pv{pos}!(\tikztotarget)!\pv{height}!270:(\tikztotarget)$)
    and ($(\tikztostart)!1-\pv{pos}!(\tikztotarget)!\pv{height}!270:(\tikztotarget)$)
    .. (\tikztotarget)\tikztonodes}},
    settings/.code={\tikzset{quiver/.cd,#1}
        \def\pv##1{\pgfkeysvalueof{/tikz/quiver/##1}}},
    quiver/.cd,pos/.initial=0.35,height/.initial=0}

\theoremstyle{plain}
\newtheorem{theorem}{Theorem}[section]
\newtheorem{lemma}[theorem]{Lemma}
\newtheorem{proposition}[theorem]{Proposition}

\theoremstyle{definition}
\newtheorem{definition}[theorem]{Definition}
\theoremstyle{remark}
\newtheorem{remark}[theorem]{Remark}
\newtheorem{example}[theorem]{Example}

\renewcommand{\epsilon}{\varepsilon}
\renewcommand{\phi}{\varphi}

\newcommand{\N}{\mathbb{N}}

\newcommand{\Q}{\mathbb{Q}}

\newcommand{\Set}{\mathbf{Set}}

\newcommand{\Frm}{\mathbf{Frm}}
\newcommand{\Loc}{\mathbf{Loc}}

\newcommand{\Topos}{\mathbf{Topos}}
\newcommand{\LocGrpd}{\mathbf{LocGrpd}}
\newcommand{\LH}{\mathbf{LH}}

\newcommand{\id}{\mathrm{id}}
\newcommand{\powerset}{\mathcal{P}}
\DeclareMathOperator{\Hom}{Hom}
\DeclareMathOperator{\Sh}{Sh}
\renewcommand{\O}{\mathcal{O}}
\newcommand{\Sub}{\mathrm{Sub}}

\newcommand{\cat}{\mathcal{C}}
\newcommand{\topos}{\mathcal{E}}
\newcommand{\ftopos}{\mathcal{F}}

\newcommand{\theory}{\mathbb{T}}

\newcommand{\PQtheory}{N \cdot \mathcal{PQ}_\mathbb{N}}
\newcommand{\PQtheorysing}{\mathcal{PQ}_\mathbb{N}}
\newcommand{\NOtheory}{N \cdot \mathbb{O}}
\newcommand{\Otheory}{\mathbb{O}}
\newcommand{\Mod}{\mathrm{Mod}}
\newcommand{\Desc}[2]{\mathbf{Desc}_{{#1}}({#2}_\bullet)}
\newcommand{\class}[1]{\llbracket {#1} \rrbracket}

\title{The representing localic groupoid for a geometric theory}

\author[G. Manuell]{Graham Manuell}
\address{Department of Mathematical Sciences, Stellenbosch University, South Africa}
\email{graham@manuell.me}
\thanks{The first author acknowledges financial support from the Centre for Mathematics of the University of Coimbra (UIDB/00324/2020, funded by the Portuguese Government through FCT/MCTES)}

\author[J. L. Wrigley]{Joshua L. Wrigley}
\address{DISAT, Univeristy of Insubria, Como, Italy}
\email{josh.l.wrigley@gmail.com}
\thanks{The second author acknowledges the financial support of the Insubria-Huawei studentship into ``Grothendieck toposes for information and computation''}

\date{April 2024}


\begin{document}

\begin{abstract}
 We give an expository, and hopefully approachable, account of the Joyal--Tierney result that every topos can be represented as a topos of sheaves on a localic groupoid. We give an explicit presentation of a representing localic groupoid for the classifying topos of a given geometric theory and discuss links with the topological groupoids of Forssell.
\end{abstract}

\maketitle
\thispagestyle{empty}

\setcounter{section}{-1}
\section{Introduction}

In \cite{joyal1984galois} Joyal and Tierney famously proved that every topos can be represented as a topos of equivariant sheaves on a localic groupoid. This provides a helpful perspective from which to understand topos theory. However, possibly due to the level of abstraction involved in the proof, many people \cite{se2728171,mo84197,mo102652} seem to be unsure of exactly how to construct such a localic groupoid in concrete cases.
The aim of this paper is to show that it is relatively straightforward to write down a localic groupoid that represents a topos directly from the geometric theory the topos classifies.

Topos theory is a powerful mathematical framework which unifies topology and logic in the language of category theory. In particular, every geometric theory has an associated \emph{classifying topos} that encodes information about the models of the theory, not only in $\Set$, but in every topos. The representation theorem of Joyal and Tierney can be understood as showing that a topos can be viewed as a topological space (in its `point-free' incarnation) together with additional automorphisms (as given by the structure of a groupoid). Thus, we might compare toposes to orbifolds from differential geometry.

While topos theory has a reputation for being difficult at times, we have attempted to make this paper approachable to non-experts. For this reason we do not include all the technical details from \cite{joyal1984galois}. Instead we try to give intuition for the essential ideas and hope to show that the main ideas behind the proof of the representation theorem are not as difficult as one might fear.

The explicit localic groupoid we construct is slightly different from the one found in \cite{joyal1984galois}, instead coinciding with the construction briefly sketched in \cite[Remark C5.2.8(c)]{Elephant} (see \cref{rem:multipleopencovers}).  Another exposition of the same construction from the perspective of mathematical logic is given in Paper III of Lindberg's Ph.D.\ thesis \cite{lindberg}.\footnote{Though the paper is not (yet) publicly available.} Our paper focuses on giving a hands-on procedure for turning a geometric theory into a presentation for the corresponding localic groupoid, which has only been implicit in prior work.

In \cref{sec:background} we provide a brief introduction to the concepts needed to understand the paper. In particular, we discuss locales, internal groupoids, classifying toposes and sheaves.

In \cref{sec:syndecr} we give an explicit description for a localic groupoid associated to the classifying topos of a geometric theory by coding models as subquotients of $\N$. It is perhaps surprising that the we do not need to consider any larger models, but this is a consequence of the localic nature of the construction.

An overview of the proof of the representation theorem from \cite{joyal1984galois} is described in \cref{sec:joyaltierney}. In \cref{sec:proof} we show how this construction yields the localic groupoid described in \cref{sec:syndecr}.

The remainder of \cref{sec:proof} is devoted to applications of our explicit description of the representing groupoid.  We calculate explicit descriptions of the left adjoints to the source and target maps in \cref{sec:leftadjoints}, and use these to demonstrate in \cref{sec:isoclasses} that the opens in the `locale of isomorphism classes' are precisely the sentences of the theory (up to provable equivalence).

In \cref{sec:countable}, we show that the localic groupoid we construct is spatial when the theory is countable.  Some readers may prefer working with topological groupoids to locales. We compare the resulting topological groupoid to the construction found in \cite{forssell2012}.

\section{Background}\label{sec:background}

We begin by recalling some background information.

\subsection{Locales}

A topological space is given by a set of points and a lattice of open sets. In the pointfree approach to topology a space is described by its lattice of open sets alone and the points are derived from this.
We will give a brief introduction to this approach. For more details see \cite[Chapters II and III]{PicadoPultr}.

\begin{definition}
 A \emph{frame} is a complete lattice satisfying the distributivity law \[a \wedge \bigvee_\alpha b_\alpha = \bigvee_\alpha a \wedge b_\alpha.\]
\end{definition}
Frames are (infinitary) algebraic structures with constants $0$ and $1$, a binary operation $\wedge$, and a join operation $\bigvee$ for each cardinality. Frame homomorphisms are maps that preserve finite meets and arbitrary joins.

Note that a topological space is simply a set $X$ together with a subframe of the powerset $\powerset{X}$. A continuous map of topological spaces induces a frame homomorphism between these frames of open sets \emph{in the opposite direction} by taking preimages. In general, we call elements of a frame \emph{opens}.

The category $\Loc$ of \emph{locales} is the opposite of the category $\Frm$ of frames and frame homomorphisms. A locale is simply a frame, but the direction of locale morphisms emphasises their geometric nature by agreeing with the direction of continuous maps of topological spaces. We will maintain a notational distinction between a locale $X$ and its frame of opens $\O X$.
If $f\colon X \to Y$ is a locale morphism, we write $f^*\colon \O Y \to \O X$ for the corresponding frame homomorphism.
Since $f^*$ preserves arbitrary joins, it has a right adjoint $f_*\colon \O X \to \O Y$.

A \emph{point} of a locale $X$ is given by a locale morphism from the terminal locale $1$ (represented by the one-point space) to $X$. In good cases, the abstract points of a locale obtained from a topological space recover the concrete points of the space itself. On the other hand, not every locale arises from a topological space (see \cref{ex:partial_surjections_from_N}). A locale that does come from a topological space is called \emph{spatial}.

As with other algebraic structures, frames can be presented by generators and relations. Such presentations can also be given a logical interpretation as the Lindenbaum--Tarski algebras for a certain kind of infinitary propositional logic called \emph{geometric logic}, which has finite conjunctions and infinitary disjunctions.
Let us consider an example.
\begin{example}[The Dedekind reals]\label{ex:dedekindreals}
Recall that a Dedekind cut on $\Q$ is given by a pair $(L,U)$ of subsets of $\Q$ satisfying certain axioms. Such a pair represents a (unique) real number that is larger than the rationals in the `lower cut' $L$ and smaller than the  rationals of the `upper cut' $U$.  
The theory of Dedekind cuts can be expressed in geometric logic by taking an atomic proposition with the (suggestive, but purely formal) name $[p \in L]$ for each $p \in \Q$, an atomic proposition $[p \in U]$ for each $p \in \Q$ and the following axioms.
\begin{displaymath}
\begin{array}{r@{\hspace{1.5ex}}c@{\hspace{1.5ex}}l@{\quad}@{}l@{\qquad\quad}r@{}}
  {[q \in L]} &\vdash& {[p \in L]} & \text{ for $p \le q$} & \text{($L$ downward closed)} \\
  {[q \in L]} &\vdash& \bigvee_{p > q} {[p \in L]} & \text{ for $q \in \Q$} & \text{($L$ rounded)} \\
   &\vdash& \bigvee_{q \in \Q} {[q \in L]} && \text{($L$ inhabited)} \\
  {[p \in U]} &\vdash& {[q \in U]} & \text{ for $p \le q$} & \text{($U$ upward closed)} \\
  {[q \in U]} &\vdash& \bigvee_{p < q} {[p \in U]} & \text{ for $q \in \Q$} & \text{($U$ rounded)} \\
   &\vdash& \bigvee_{q \in \Q} {[q \in U]} && \text{($U$ inhabited)} \\
  {[p \in L]} \land {[q \in U]} &\vdash& \bot & \text{ for $p \ge q$} & \text{($L$ and $U$ disjoint)} \\
   &\vdash& {[p \in L]} \lor {[q \in U]} & \text{ for $p < q$} & \text{(locatedness)}
\end{array}
\end{displaymath}
Note a sequent $\phi \vdash \psi$ is interpreted as saying that $\psi$ holds whenever $\phi$ does. If $\phi$ is missing it is understood to be $\top$ (i.e.\ true). Sequents are necessary since the logic does not contain an implication connective.
A model of such a theory assigns a truth value to each basic proposition such that the sequents are satisfied. In this case, such a model gives a Dedekind cut $(L,U)$ where $L$ is the set of $p \in \Q$ for which $[p \in L]$ is true and $U$ is the set of $p \in \Q$ for which $[p \in U]$ is true.
\end{example}

A propositional geometric theory yields a frame presentation by simply taking the basic propositions to be generators and each axiom $\phi \vdash \psi$ to be a relation $\phi \le \psi$ (or the equivalent equation $\phi \wedge \psi = \phi$). We say that the corresponding locale \emph{classifies} the geometric theory. The universal property of the presentation ensures that models of the geometric theory correspond to points of the locale, since the $\O 1$ is the frame of truth values $\{0,1\}$.
The classifying locale for the theory of Dedekind cuts on $\Q$ is the locale of real numbers (with their usual topology). The propositions $[p \in L]$ and $[p \in U]$ correspond to the opens $(p, \infty)$ and $(-\infty, p)$ respectively.

As a second example consider the following more unusual theory.
\begin{example}[Partial surjections from $\N$ to $X$]\label{ex:partial_surjections_from_N}
Fix a set $X$ and consider the following geometric theory of partial surjections from $\N$ to $X$. There is a basic proposition denoted by $[f(n) = x]$ for each $n \in \N$ and $x \in X$, which is of course interpreted to mean that the partial function maps $n$ to $x$. The axioms are as follows.\footnote{For simplicity we have assumed classical logic for this definition. Constructively, the functionality axiom would instead be $[f(n) = x] \land [f(n) = y] \vdash \bigvee \{\top \mid x = y\}$ for all $x,y \in X$.}
\begin{displaymath}
\begin{array}{r@{\hspace{1ex}}c@{\hspace{1ex}}l@{\quad}@{}l@{\quad}r@{}}
  [f(n) = x] \land [f(n) = y] &\vdash& \bot & \text{for $n \in \N$, $x,y \in X$, $x \neq y$} & \text{(functionality)} \\
   &\vdash& \bigvee_{n \in \N} [f(n) = x] & \text{for $x \in X$} & \text{(surjectivity)}
\end{array}
\end{displaymath}

If $X$ is countable, the resulting locale is not so strange. However, if $X$ is $\powerset \N$, say, then there are no surjections from $\N$ to $X$. Nonetheless, the locale is nontrivial (see \cite[Example C1.2.8]{Elephant}). Thus, this locale is wildly non-spatial.
\end{example}

Many notions from topology have analogues for locales. In particular, we can define a notion of open locale maps.
\begin{definition}
  A locale morphism $f\colon X \to Y$ is \emph{open} if its associated frame homomorphism $f^*\colon \O Y \to \O X$ has a left adjoint $f_!\colon \O X \to \O Y$ that satisfies the so-called \emph{Frobenius reciprocity} condition:
  $f_!(f^*(u) \wedge v) = u \wedge f_!(v)$.
  
  If $f^* f_! = \id_{\O X}$ (or equivalently $f^*$ is surjective), we say $f$ is an \emph{open sublocale inclusion}.
  If $f_! f^* = \id_{\O Y}$ (or $f^*$ is injective), we say $f$ is an \emph{open surjection}.
\end{definition}

The left adjoint can be understood as giving the images of opens of $X$ under $f$. Open sublocales of $X$ are in bijection with the elements of the frame $\O X$.
Open maps of locales are stable under composition and pullback, and pulling back open sublocales along a map $h$ agrees with the action of the frame homomorphism $h^*$.

\begin{definition}
  A locale map $f\colon X \to Y$ is a \emph{local homeomorphism} if it is open and so is the `diagonal' map $\delta\colon X \to X \times_Y X $, whose codomain is given by the pullback
  \[
  \begin{tikzcd}
      X \times_Y X \ar{r}{\pi_1} \ar{d}[']{\pi_2} \arrow[dr, phantom, "\lrcorner", very near start]
 & X \ar{d}{f} \\
      X \ar[swap]{r}{f} & Y,
  \end{tikzcd}
  \]
  and which satisfies $\pi_1 \delta = \pi_2 \delta = \id_X$.
\end{definition}
It is easy to see that open sublocale inclusions are local homeomorphisms.
Local homeomorphisms can also be equivalently defined in terms of restricting to open embeddings on an open cover of the domain, in a similar way to how they often are for topological spaces.

\subsection{Topos theory}

It is difficult to summarise what topos theory \emph{is} due to the plethora of perspectives on the subject (the eponymous `sketches of an elephant' of \cite{Elephant}).  A (Grothendieck) topos can either be defined as a category satisfying the abstract \emph{Giraud axioms} (see \cite[Theorem A1.1]{SGL}), or the category of \emph{sheaves on a site} --- that is, a category $\topos$ which embeds as a left-exact-reflective subcategory of some category of presheaves $\Set^{\cat^{op}}$ (i.e.\ a subcategory whose inclusion has a finite-limit-preserving left adjoint $\Set^{\cat^{op}} \to \topos$). The specific embedding $\topos \hookrightarrow \Set^{\cat^{op}}$ is not included as part of the defining data of $\topos$. Indeed, many different sites can present the same topos.

On the surface, the formal definitions of a topos do not appear that exciting.
However, the many desirable properties possessed by toposes lend themselves to other perspectives on the subject.  For example, toposes behave as mathematical universes --- they have a powerful internal language that can interpret constructive mathematics.

For this paper, two aspects of topos theory will prove important: topos theory as a syntax invariant approach to model theory (discussed in \cref{sec:classifying}), and topos theory as a generalisation of locale theory.

\begin{example}[Sheaves on a space]
A fundamental example of a topos is the \emph{topos of sheaves on a locale} $X$, denoted by $\Sh(X)$.  This is the slice category $\LH/X$, where $\LH \subseteq \Loc$ is the category of locales and local homeomorphisms.  A topos of the form $\Sh(X)$ is said to be \emph{localic}.
\end{example}

Given a topos $\topos$ and an object $E \in \topos$, the subobjects of $E$ (i.e.\ equivalence classes of monomorphisms $U \hookrightarrow E$) form a frame $\Sub_\topos(E)$ (see \cite[Proposition III.8.1]{SGL}).  For each arrow $g \colon E \to E'$ of $\topos$, the map $g^\ast \colon \Sub_\topos(E') \to \Sub_\topos(E)$, given by taking pullbacks of subobjects along $g$, is a frame homomorphism (see \cite[Proposition III.8.2]{SGL}).

The morphisms of toposes we consider are \emph{geometric morphisms}.  A geometric morphism between two toposes $f \colon \ftopos \to \topos$ is an adjoint pair of functors
\[\begin{tikzcd}
	\ftopos & \topos
	\arrow[""{name=0, anchor=center, inner sep=0}, "{f_\ast}"', shift right=2, from=1-1, to=1-2]
	\arrow[""{name=1, anchor=center, inner sep=0}, "{f^\ast}"', shift right=2, from=1-2, to=1-1]
	\arrow["\dashv"{anchor=center, rotate=-90}, draw=none, from=1, to=0]
\end{tikzcd}\]
such that the left adjoint $f^\ast$ preserves finite limits (in addition to colimits).  The left adjoint is commonly called the \emph{inverse image} functor, while the right adjoint is called the \emph{direct image} functor.

\begin{example}
    If $\topos$ is a left-exact-reflective subcategory of $\Set^{\cat^{op}}$, then the adjoint pair
    \[\begin{tikzcd}
	\topos & \Set^{\cat^{op}}
	\arrow[""{name=0, anchor=center, inner sep=0}, hook, shift right=2, from=1-1, to=1-2]
	\arrow[""{name=1, anchor=center, inner sep=0}, shift right=2, from=1-2, to=1-1]
	\arrow["\dashv"{anchor=center, rotate=-90}, draw=none, from=1, to=0]
\end{tikzcd}\]
is a geometric morphism.
\end{example}

Each locale morphism $f \colon X \to Y$ induces a geometric morphism $\Sh(f) \colon \Sh(X) \to \Sh(Y)$, whose inverse image part we write as $f^*$. It sends sends a local homeomorphism $q \colon W \to Y$ to its pullback along $f$
\[
\begin{tikzcd}
    f^\ast(W)  \arrow[dr, phantom, "\lrcorner", very near start] \ar{r} \ar{d} & W \ar{d}{q} \\
    X \ar[swap]{r}{f} & Y,
\end{tikzcd}
\]
and a morphism $g$ of $\Sh(Y)$ to the induced map
\[
\begin{tikzcd}[row sep = small, column sep = small]
    f^\ast(W') \ar[dashed]{dr}{f^\ast(g)} \ar{rr} \ar[bend right]{rddd} && W' \ar{dr}{g} \ar[bend right]{rddd} & \\
         & f^\ast(W) \ar[crossing over]{rr}  \arrow[ddrr, phantom, "\lrcorner",pos= 0 ] \ar{dd} && W \ar{dd}{q} \\
    \\
    & X \ar[swap]{rr}{f} && Y.
\end{tikzcd}
\]
Note that this agrees with the definition of $f^*$ as a frame homomorphism if we conflate opens and open sublocales.

In fact, locale theory can be reinterpreted inside topos theory via the full and faithful functor
$\Sh \colon \Loc \hookrightarrow \Topos$
(see \cite[\S IX]{SGL}).  The functor $\Sh$ has a left adjoint, the \emph{localic reflection}, which sends a topos $\topos$ to the frame $\Sub_\topos(1)$ of subobjects of the terminal object of $\topos$ (see \cite[Proposition A4.6.12]{Elephant}).

\subsection{Properties of geometric morphisms}\label{sec:propertiesofgeomorph}

Many properties of locale morphisms generalise to properties of geometric morphisms.  For example, a morphism $f \colon X \to Y$ of locales is open (and surjective) if and only if the corresponding geometric morphism $\Sh(f) \colon \Sh(X) \to \Sh(Y)$ is open (and surjective) in the following sense.

\begin{definition}
    A geometric morphism $f \colon \ftopos \to \topos$ is \emph{open} if, for each object $E \in \topos$, the induced frame homomorphism on subobjects
    \[f^\ast_E \colon \Sub_\topos (E) \to \Sub_\ftopos (f^\ast(E))\]
    has a left adjoint $f_!^E$ and this left adjoint is natural in $E$ in the sense that, for each arrow $g \colon E \to E'$, the square
    \[
    \begin{tikzcd}
        \Sub_\topos(E')  \ar{d}[']{\Sub_\topos(g)} & \Sub_\ftopos(f^\ast(E')) \ar{l}[']{f_!^{E'}} \ar{d}{\Sub_\ftopos(f^\ast(g))} \\
        \Sub_\topos(E) & \Sub_\ftopos(f^\ast(E)) \ar{l}[']{f_!^E}
    \end{tikzcd}
    \]
    commutes. (Note that in particular, by choosing $g \colon E \hookrightarrow E'$ to be a monomorphism, we can show that $f_!^{E'}$ satisfies Frobenius reciprocity and so $f^*_{E'}$ is open.) Moreover, the open geometric morphism $f$ is said to be \emph{surjective} if $f^\ast$ is a faithful functor.
\end{definition}
As is the case for the analogous class of locale morphisms, open (surjective) geometric morphisms are stable under pullback (see \cite[Proposition VII.1.3]{joyal1984galois} or \cite[Theorem 4.7]{johnstone1980}).

The geometric morphism $\Sh(f) \colon \Sh(X) \to \Sh(Y)$ induced by a locale map $f \colon X \to Y$ has the property that every object in $\Sh(X)$ is a \emph{subquotient} of something in the inverse image --- that is, for all $F \in \Sh(X)$ there is a diagram:
\[
\begin{tikzcd}
    F & S \ar[two heads]{l} \ar[hook]{r} & f^\ast(E).
\end{tikzcd}
\]
Such geometric morphisms are called \emph{localic}, so-named because a topos $\topos$ is localic if and only if the (necessarily unique) geometric morphism $\gamma \colon \topos \to \Set$ is a localic geometric morphism.

Let us see why $\gamma \colon \Sh(X) \to \Set$ is a localic geometric morphism for a locale $X$. Each object $q \colon Y \to X$ of $\Sh(X)$ is a local homeomorphism, and so $Y$ is covered by a collection $(s_i)_{i \in I}$ of local sections of $q$,
\[
\begin{tikzcd}
    & Y \ar{d}{q} \\
    U_i \ar[hook]{r} \ar{ru}{s_i} & X.
\end{tikzcd}
\]
Thus, in $\Sh(X)$ there is a diagram
\[
\begin{tikzcd}
    Y \ar[bend right]{rd} & \ar[two heads]{l}[']{\overline{s}} \coprod_{i \in I} U_i \ar[hook]{r} \ar{d} & \coprod_{i \in I} X \ar[bend left = 1.9em]{ld} &[-35pt] \cong \gamma^\ast(I)  \\
    & X &
\end{tikzcd}
\]
as required.

We highlight two important facts about localic geometric morphisms.
\begin{enumerate}
    \item Localic geometric morphisms are stable under pullback in $\Topos$ (see \cite[Proposition 2.1]{johnstone1981}).
    \item Localic geometric morphisms are closed under composition (see \cite[Lemma 1.1]{johnstone1981}). So if $h \colon \mathcal{H}' \to \mathcal{H}$ is a localic geometric morphism and $\mathcal{H}$ is a localic topos, then $\mathcal{H}'$ is a localic topos too.
\end{enumerate}

Geometric morphisms into toposes can be specified by internal structures in the codomain topos.
A geometric morphism $f \colon \ftopos \to \topos$ is localic if and only if $\ftopos$ is the topos of \emph{internal sheaves for an internal locale} (see \cite[Theorem 5.34]{johnstone1977}).  We won't seek to make sense of the phrase `internal sheaves for an internal locale' here, but a precise formulation can be found in \cite{caramello2022} or \cite{wrigley2023}.

\subsection{Classifying toposes}\label{sec:classifying}

Propositional geometric logic has a first-order generalisation which involves not just propositions, but also sorts and relation symbols. First-order geometric logic is the fragment of infinitary predicate logic that includes finite conjunction, infinitary disjunction, an equality predicate and existential quantification, i.e.\ the symbols $\land$, $\bigvee$, $=$ and $\exists$. Furthermore, sequents are now equipped with \emph{contexts} which, at a minimum, contain the free variables of the formulae inside a sequent (equipped with their types).

\begin{remark}
  First-order geometric theories are often defined to also allow function symbols. However, these can always be defined as binary relations together with functionality and totality axioms. We will nonetheless occasionally find it convenient to use function notation for such relations.
\end{remark}

\begin{example}\label{ex:linear_order}
An example of such a theory is that of inhabited total orders. This consists of a single sort $X$ together with a binary relation ${\le} \subseteq X \times X$ satisfying the following axioms.
\begin{displaymath}
\begin{array}{r@{\hspace{1.5ex}}l@{\hspace{1.5ex}}l@{\quad}@{\quad}r@{}}
  &\vdash_{x,y,z\colon X}& x \le  x & \text{(reflexivity)} \\
  x \le y \land y \le z &\vdash_{x,y,z\colon X}& x \le  z & \text{(transitivity)} \\
  x \le y \land y \le x &\vdash_{x,y\colon X}& x = y & \text{(antisymmetry)} \\
  &\vdash_{x,y\colon X}& x \le y \lor y \le x  & \text{(totality)} \\
  & \vdash_{\emptyset} & \exists x \colon X. \ \top & \text{(inhabitedness)}
\end{array}
\end{displaymath}
Note that, just as with \cref{ex:dedekindreals}, we must use sequents since geometric logic does not contain an implication or universal quantification symbol.  A sequent $\varphi \vdash_{\vec{x}} \psi$ is understood as expressing ``for all $\vec{x}$, $\varphi$ implies $\psi$''.
\end{example}

Propositional geometric theories can be understood as the special case of first-order geometric theories where there are no sorts. Basic propositions are simply understood as nullary relations.

Just as there is a classifying locale for every propositional geometric theory, there is a \emph{classifying topos} for a general geometric theory (see \cite[Proposition D3.1.12]{Elephant} or \cite[Theorem 2.1.10]{TST}). The classifying topos for a geometry theory $\theory$ is written $\Set[\theory]$ and satisfies the universal property
\[\Hom(-, \Set[\theory]) \cong \Mod_{\theory}(-)\]
where $\Mod_{\theory}(\topos)$ denotes the category of $\theory$-models in the topos $\topos$. Moreover, every topos classifies some geometric theory (see \cite[Remark D3.1.13]{Elephant} or \cite[Theorem 2.1.11]{TST}). A $\theory$-model in an arbitrary topos $\ftopos$ consists of an object $X^M$ for each sort $X$ and, for each relation symbol, a subobject $R^M \hookrightarrow X_1^M \times \, \dots \, \times X_n^M$ such that the axioms of $\theory$ are satisfied in the following sense.  From our basic relation symbols, we can construct the interpretation of each geometric formula, and a sequent $\varphi \vdash_{\vec{x}} \psi$ is satisfied if $\varphi^M \le \psi^M$ as subobjects of $X_1^M \times \, \dots \, \times X_n^M$.  See \cite[\S D1]{Elephant} for more details.

The topos $\Set$ classifies the theory with no sorts, no symbols, and no axioms.  If $\theory$ is a propositional theory (i.e.\ there are no sorts), then the classifying topos $\Set[\theory]$ is simply the topos of sheaves on the classifying locale of $\theory$.  Thus, the geometric morphism $\Set[\theory] \to \Set$ is localic.

More generally, if $\theory$ is a theory with $N$ sorts, then $\Set[\theory]$ is the topos of sheaves for an internal locale of $\Set[\NOtheory]$ --- that is, there exists a localic geometric morphism $L \colon \Set[\theory] \to \Set[\NOtheory]$, where $\NOtheory$ denotes the first-order theory with $N$ sorts, no relations or functions, and no axioms. This is demonstrated for $N = 1$ in \cite[Theorem D3.2.5]{Elephant} (this appears in \cite{joyal1984galois} in entirely categorical terms as Proposition VII.3.1) or, for arbitrary $N$, in \cite[Definition 7.1.1 \& Theorem 7.1.3]{TST}. Thus, in a certain expanded sense, every topos is a `localic' topos\footnote{Amusingly, the slogan \emph{all topoi are localic} was used by Freyd as the title of his paper \cite{freyd}, but for a \emph{different} sense in which every topos is `localic'!}.

\subsection{Equivariant sheaves on a groupoid}
A localic groupoid is a groupoid internal to $\Loc$, just as a topological group is an internal group in the category of topological spaces.

\begin{definition}
A \emph{localic groupoid} $G$ is a diagram in $\Loc$ of the form
    \[\begin{tikzcd}
 G_1 \times_{G_0} G_1 \ar[shift left = 4]{r}{\pi_{2}} \ar{r}{m} \ar[shift right = 4]{r}{\pi_{1}} & G_1  \ar[loop, distance=2em, in=305, out=235, "i"'] \ar[shift left = 4]{r}{t} \ar[shift right = 4]{r}{s} & \ar{l}[']{e} G_0 ,
\end{tikzcd}\]
such that the equations
\[s \circ e = t \circ e = \id_{G_0},\]
\[s \circ m = s \circ \pi_1, \ \ t \circ m = t \circ \pi_2,\]
\[m \circ (\id_{G_1} \times_{G_0} m) = m \circ (m \times_{G_0} \id_{G_1}), \]
\[m \circ (e \circ s, \id_{G_1}) = \id_{G_1} = m \circ (\id_{G_1}, e \circ t),\]
\[s \circ i = t, \ \ t \circ i = s, \]
\[m \circ (\id_{G_1}, i) = e \circ s, \ \  m \circ (i, \id_{G_1}) = e \circ t\]
all hold.
\end{definition}

Intuitively, these conditions express that $G_0$ is the \emph{locale of objects} and $G_1$ is the \emph{locale of arrows} of a category in which every arrow is invertible. The `source' map $s$ and `target' map $t$ assign arrows to their domain and codomain respectively. The map $e$ picks out the identity arrow of an object. The map $m$ gives the composites of composable pairs, while $i$ yields the inverse of each arrow.  The equations imposed on a localic groupoid express this interpretation, e.g. the equation $s \circ e = t \circ e = \id_{G_0}$ says that the source and target of the identity arrow on an object $x \in G_0$ is $x$, as we would expect.  Of course, a similar definition in $\Set$ would give the usual notion of a small groupoid.

\begin{definition}
A localic groupoid is said to be \emph{open} if $s$ and $t$ are both open maps.  We note that since $s \circ i = t$ and $i$ is an isomorphism, $t$ is open if and only if $s$ is.
\end{definition}

\begin{example}\label{ex:grpdexs}
Let us consider four important classes of examples of localic groupoids.
\begin{enumerate}[start=0]
    \item Every small groupoid gives localic groupoid by viewing the sets of objects and morphisms as spaces with the discrete topology. We might call these \emph{topologically discrete} groupoids.
    \item For each locale $X$,
    \begin{equation}\label{example1}
        \begin{tikzcd}
    X \ar{r}{\id_X} \ar[shift left = 4]{r}{\id_X} \ar[shift right = 4]{r}{\id_X} & X \arrow[loop,  "\id_X "', distance=2em, in=305, out=235] \ar[shift left = 4]{r}{\id_X} \ar[shift right = 4]{r}{\id_X} & X \ar{l}[']{\id_X}
    \end{tikzcd}
    \end{equation}
    is a localic groupoid. This can be viewed as the `discrete' category on the locale of objects $X$. We call this a \emph{categorically discrete} groupoid.
    \item Let $(G,m,e)$ be a localic group.  The diagram
    \begin{equation}\label{example2}
        \begin{tikzcd}
G \times G \ar[shift left = 4]{r}{\pi_1} \ar{r}{m} \ar[shift right = 4]{r}{\pi_2} & G  \arrow[loop,  "i "', distance=2em, in=305, out=235] \ar[shift left = 3]{r}{!} \ar[shift right = 3]{r}{!} & 1 \ar{l}[']{e},
\end{tikzcd}
    \end{equation}
defines a localic groupoid. This is exactly like how a group in $\Set$ can be viewed as a one-object category.
    \item If $G$ is a localic group acting continuously on a locale $X$ by $\alpha \colon G \times X \to X$, then the diagram
    \begin{equation}\label{example3}
        \begin{tikzcd}
G \times G \times X \ar[shift left = 5]{r}{\pi_{1,3}} \ar{r}{m \times \id_X } \ar[shift right = 4]{r}{\pi_{2,3}} & G \times X  \arrow[loop,  "i \times \id_X "', distance=2em, in=305, out=235] \ar[shift left = 5]{r}{\alpha} \ar[shift right = 4]{r}{\pi_2} & X \ar{l}[']{(e, \id_X)}.
\end{tikzcd}
    \end{equation}
is a localic groupoid.
\end{enumerate}
\end{example}

A \emph{sheaf} $Y$ for a localic groupoid $G$ is a local homeomorphism $q \colon Y \to G_0$ together with a compatible $G_1$-action (though we will often omit the map $q$ and the action from our notation). A $G_1$-action is a locale map $\beta \colon Y \times_{G_0} G_1 \to Y$, where $Y \times_{G_0} G_1$ is the pullback of $q$ and the source morphism $s\colon G_1 \to G_0$, such that the equations
\begin{align*}
  q(\beta(y,g)) &= t(g), \\
  \beta(\beta(y,g),h) &= \beta(y,m(g,h)), \\
  \beta(y,e(q(y))) &= y
\end{align*}
hold in the internal logic of $\Loc$.
Note that although the locales $Y$, $G_1$ and $G_0$ may not be spatial, we are able to reason in a suggestive `point-set' theoretic manner.  This is explained further in \cref{sec:points} below.

A \emph{morphism of sheaves} is an `equivariant morphisms of bundles' --- that is, a locale morphism $f \colon Y \to Y'$ such that
\[q'(f(y)) = q(y) \ \text{ and }\ f(\beta(y,g)) = \beta'(f(y),g)\]
hold in the internal logic.

\begin{definition}
We denote the category of sheaves and morphisms of sheaves on a localic groupoid $G$ by $\Sh(G)$.
\end{definition}

The category $\Sh(G)$ is a topos by \cite[Proposition 5.2]{moerdijk1988}. We say that an arbitrary topos $\topos$ is \emph{represented} by the groupoid $G$ if there is an equivalence $\topos \simeq \Sh(G)$.

\begin{example}\label{exsrevisited}

We revisit the example groupoids of \cref{ex:grpdexs} and describe the resulting sheaf toposes.
\begin{enumerate}[start=0]
    \item The topos of sheaves on a small groupoid $G$ (viewed as a topologically discrete localic groupoid) is essentially the category of discrete opfibrations over $G$ and is therefore equivalent to the functor category $\Set^G$.
    \item The topos of sheaves on the groupoid (\ref{example1}) is the familiar topos of sheaves $\Sh(X)$ on $X$.
    \item The topos of sheaves on the groupoid (\ref{example2}) is the topos $\mathbf{B} G$ of discrete sets with a continuous action by $G$ and equivariant maps between these. (See \cite[\S III.9]{SGL} for description in terms of topological groups. The localic case is similar.)
    \item The topos of sheaves on the groupoid (\ref{example3}) is the topos of $G$-equivariant sheaves over $X$, as seen in \cite{freyd} and \cite[Proposition A4.6]{SGL}.
\end{enumerate}
\end{example}

The objects and morphisms of $\Sh(G)$ can be given a more compact definition in terms of \emph{descent data} (the reasons for the nomenclature will become apparent in \cref{sec:descent}). A descent datum for $G$
is a pair consisting of a local homeomorphism $q \colon Y \to G_0$ and a morphism $\theta \colon s^\ast(Y) \to t^\ast(Y)$ such that
\[e^\ast(\theta) = \id_Y \text{ and } m^\ast(\theta) = \pi_2^\ast(\theta) \circ \pi_1^\ast(\theta).\] 
A morphism of descent data $f \colon (Y,\theta) \to (Y',\theta')$ is a commuting triangle
\[
\begin{tikzcd}[column sep = tiny]
    Y \ar{rd}[']{q}\ar{rr}{f} && Y' \ar{ld}{q'} \\
    & G_0 &
\end{tikzcd}
\]
(i.e.\ a morphism $f \colon Y \to Y'$ in $\Sh(G_0)$) such that the square
\[
\begin{tikzcd}
    s^\ast(Y) \ar{d}[']{s^\ast(f)} \ar{r}{\theta} & t^\ast(Y) \ar{d}{t^\ast(f)} \\
    s^\ast(Y') \ar[swap]{r}{\theta'} & t^\ast(Y')
\end{tikzcd}
\]
commutes.

That the two definitions of sheaves on $G$ are equivalent is a matter of unravelling definitions. The notational difference arises because, for descent data, we keep track of the arrow $f \in G_1$ once it has been applied to a point $y \in Y$.  Indeed, given a $G_1$-action $\beta \colon Y \times_{G_0} G_1 \to Y$, the corresponding descent datum is the map $\theta_\beta$ that sends the pair $(y,f) \in s^\ast(Y)$ to $(\beta(y,f),f) \in t^\ast(Y)$.  For completeness, we explain the equivalence between $G_1$-actions and descent data in detail in \cref{app:equivanddesc}.

\subsection{Morphisms of localic groupoids}

As established in \cite[\S 5.4]{moerdijk1988}, taking sheaves on a localic groupoid is a functorial construction with respect to homomorphisms of localic groupoids. Homomorphisms of localic groupoids are functors between internal categories.

\begin{definition}
    A \emph{homomorphism of localic groupoids} $\varphi \colon G \to H$ is a pair of locale morphisms $\varphi_0 \colon G_0 \to H_0$ and $\varphi_1 \colon G_1 \to H_1$, between the locales of objects and arrows respectively, which commute with the respective structure morphisms of the groupoids.
    \begin{equation}\label{eq:homogrpd}
        \begin{tikzcd}
	{G_1 \times_{G_0}G_1} & {H_1 \times_{H_0} H_1} \\
	{G_1} & {H_1} \\
	{G_0} & {H_0}
	\arrow["{\varphi_0}", from=3-1, to=3-2, swap]
	\arrow["{\varphi_1}", from=2-1, to=2-2]
	\arrow["s"', shift right=3, from=2-1, to=3-1]
	\arrow["t", shift left=3, from=2-1, to=3-1]
	\arrow["t'", shift left=3, from=2-2, to=3-2]
	\arrow["s'"', shift right=3, from=2-2, to=3-2]
	\arrow["e'"{description}, from=3-2, to=2-2]
	\arrow["e"{description}, from=3-1, to=2-1]
	\arrow[dashed, from=1-1, to=1-2]
	\arrow["m"', from=1-1, to=2-1]
	\arrow["m'", from=1-2, to=2-2]
\end{tikzcd}
    \end{equation}
\end{definition}

Each homomorphism of localic groupoids $\varphi \colon G \to H$ induces a geometric morphism $\Sh(\varphi) \colon \Sh(G) \to \Sh(H)$ (see \cite[\S 5]{moerdijk1988}).  The inverse image functor $\Sh(\varphi)^\ast$ sends the descent datum $(Y,\theta)$ to the pair consisting of $\varphi_0^\ast(Y)$ and the map
\[s^\ast \varphi_0^\ast(Y) = \varphi_1^\ast s'^\ast (Y) \xrightarrow{\varphi_1^\ast(\theta)} \varphi_1^\ast t'^\ast (Y) = t^\ast \varphi_0^\ast(Y).\]
That $\varphi_1^\ast(\theta)$ satisfies the required equations follows from the commutativity of \cref{eq:homogrpd}.  Each morphism $f \colon (Y,\theta) \to (Y',\theta')$ of descent data is sent by $\Sh(\varphi)^\ast$ to the map 
\[\varphi_0^\ast(f) \colon (\varphi_0^\ast(Y),\phi_1^\ast(\theta)) \to (\varphi_0^\ast(Y'),\phi_1^\ast(\theta')).\]
The required commutativity condition $\varphi_1^\ast(\theta') \circ s^\ast \phi_0^\ast (f) = t^\ast \varphi_0^\ast (f) \circ \varphi_1^\ast (\theta)$ follows, since
\begin{align*}
    \varphi_1^\ast(\theta') \circ s^\ast \phi_0^\ast (f) & = \varphi_1^\ast(\theta') \circ  \phi_1^\ast s'^\ast (f) \\
    & =  \varphi_1^\ast( \theta' \circ s'^\ast(f)) \\
    & = \varphi_1^\ast( t'^\ast(f) \circ \theta) \\
    & = \varphi_1^\ast t'^\ast(f) \circ \varphi_1^\ast(\theta) \\
    &= t^\ast \varphi_0^\ast (f) \circ \varphi_1^\ast (\theta).
\end{align*}

Thus, we can define a functor $\Sh \colon \LocGrpd \to \Topos$
from the category of localic groupoids and their homomorphisms into the category of toposes and geometric morphisms. The main result of Joyal and Tierney (see \cref{thm:joyal_tierney}) is equivalent to the statement that this functor is essentially surjective on objects.

\subsection{Reasoning using points}\label{sec:points}

As explained in \cite[\S 5.3]{moerdijk1988}, we can often express proofs in locale theory in the more familiar notation of point-set topology, provided a `point' $y \in Y$ is taken to mean a `generalised point' of $Y$, i.e.\ a map $y \colon U \to Y$.  In this case, we will call $y$ a $U$-\emph{point}.  To translate a `point-set' argument back to a concrete one, each instance of $y \in Y$ should be replaced by a generic locale morphism $y \colon U \to Y$, and the notation $f(y)$ for some map $f \colon Y \to X$ is translated as the composite $f \circ y \colon U \to Y \to X$.

We can also use generalised points of toposes, i.e.\ arbitrary geometric morphisms $f \colon \topos' \to \topos$, in order to reason about them as though they were spaces (see \cite{vickers2007locales}) --- though in this case we must also consider morphisms of points, since toposes exist at a higher categorical level than locales.  This is especially useful when combined with the theory of classifying toposes, since we can define a geometric morphism $g \colon \Set[\theory] \to \Set[\theory']$ by describing how $g$ acts on a (generalised) point $\ftopos \to \Set[\theory]$ and morphisms of these points.  That is to say, we can define $g$ by describing how it transforms a $\theory$-model (in $\ftopos$) into a $\theory'$-model and a $\theory$-model homomorphism into a $\theory'$-model homomorphism. For example, the localic geometric morphism $L \colon \Set[\theory] \to \Set[\NOtheory]$ associated with an $N$-sorted theory sends a $\theory$-model to the $N$ objects of its underlying sorts and a $\theory$-model homomorphism to the $N$ underlying functions between these objects.

This perspective lends itself well to the problem of determining the geometric theory classified by certain (bi)limits of other classifying toposes, using the method described in \cite[Proposition 8.43]{vickers2007locales}.

\begin{example}\label{ex:limitsoftoposes}
Let us consider some examples of how to compute limits with this approach.
  \begin{enumerate}
    \item\label{ex:limitsoftoposes:enum:product} Let $\theory$ and $\theory'$ be geometric theories. The data of an $\ftopos$-point of the product topos $\Set[\theory] \times \Set[\theory']$ can be defined by a pair of geometric morphisms $\ftopos \to \Set[\theory]$ and $\ftopos \to \Set[\theory']$ --- that is, a pair of a $\theory$-model and a $\theory'$-model in $\ftopos$.  Thus, we conclude that the topos $\Set[\theory] \times \Set[\theory']$ classifies the theory given by a copy of $\theory$ and a copy of $\theory'$ (over separate sorts).
    
    \item\label{ex:limitsoftoposes:enum:pullback} Let $\theory_1, \theory_2$ be \emph{localic expansions} (see \cite[Definition 7.1.1]{TST}) of a theory $\theory_3$, i.e.\ all three theories share the same sorts, but the theories $\theory_1$ and $\theory_2$ add new relation symbols and new axioms to $\theory_3$. Let $R_1 \colon \Set[\theory_1] \to \Set[\theory_3]$ be the geometric morphism that acts on (generalised) points by sending a $\theory_1$-model to its $\theory_3$-\emph{reduct}, i.e.\ the $\theory_3$-model obtained when we forget the extra structure added by $\theory_1$, and which sends a $\theory_1$-model homomorphism to its underlying homomorphism on the $\theory_3$-reducts. (This is precisely the localic geometric morphism posited by \cite[Theorem 7.1.3]{TST}.) We define $R_2 \colon \Set[\theory_2] \to \Set[\theory_3]$ in a similar fashion.
    
    An $\ftopos$-point of the (bi)pullback
    \[
    \begin{tikzcd}
        \Set[\theory_1] \times_{\Set[\theory_3]} \Set[\theory_2] \arrow[dr, phantom, "\lrcorner", very near start,xshift=-10pt] \ar{r} \ar{d} & \Set[\theory_2] \ar{d}{R_2} \\
        \Set[\theory_1] \ar[swap]{r}{R_1} & \Set[\theory_3]
    \end{tikzcd}
    \]
    consists of the data of a pair of $\ftopos$-points $M \colon \ftopos \to \Set[\theory_1]$ and $N \colon \ftopos \to \Set[\theory_2]$ and an isomorphism $R_1 \circ M \cong R_2 \circ N$.  That is, the topos $\Set[\theory_1] \times_{\Set[\theory_3]} \Set[\theory_2]$ classifies the theory whose models are a pair of a $\theory_1$-model and a $\theory_2$-model whose $\theory_3$-reducts are isomorphic.
  \end{enumerate}
\end{example}

\begin{remark}\label{rem:Tiso}
    Some readers may wonder how our theory is impacted when we vary the specific notion of 2-limit we consider. Ultimately, as classifying toposes are only defined up to equivalence, this won't be of importance provided that the notion is sufficiently weak. We will focus on comparing, for a geometric theory $\theory$, the various notions of `pullback' for the diagram
    \[
    \begin{tikzcd}
        & \Set[\theory] \ar{d}{\id_{\Set[\theory]}} \\
        \Set[\theory] \ar{r}{\id_{\Set[\theory]}} & \Set[\theory].
    \end{tikzcd}
    \]
    When calculating the bipullback as in \cref{ex:limitsoftoposes} above, we are implicitly taking the \emph{iso-comma object} of the cospan.  This is the topos $\topos$ that is universal with respect to the data of projections $r,u \colon \topos \rightrightarrows \Set[\theory]$ and an isomorphism
    \[\begin{tikzcd}
    \topos & {\Set[\theory]} \\
    {\Set[\theory]} & {\Set[\theory]}.
    \arrow["r", from=1-1, to=1-2]
    \arrow[""{name=0, anchor=center, inner sep=0}, "u"', from=1-1, to=2-1]
    \arrow["{\id_{\Set[\theory]}}"', from=2-1, to=2-2]
    \arrow[""{name=1, anchor=center, inner sep=0}, "{\id_{\Set[\theory]}}", from=1-2, to=2-2]
    \arrow["\cong"{description}, draw=none, from=0, to=1]
    \end{tikzcd}\]
    As in \cref{ex:limitsoftoposes}, we recognise that $\topos$ classifies the theory of $\theory$-model isomorphisms.  We denote this theory by $\theory_{\cong}$.  An explicit axiomatisation of this theory is given in \cref{def:locale_of_morphisms} below.
    
    Subtle changes to the notion of 2-pullback we take can change the specific presentation for the theory classified by the topos.  For example, if we instead considered the \emph{pseudo-pullback}, i.e.\ the topos $\topos'$ that is universal with respect to the data
    
    \[\begin{tikzcd}
    \topos' & {\Set[\theory]} \\
    {\Set[\theory]} & {\Set[\theory]},
    \arrow[dashed, from=1-1, to=1-2]
    \arrow[dashed, from=1-1, to=2-1]
    \arrow["{\id_{\Set[\theory]}}"', from=2-1, to=2-2]
    \arrow["{\id_{\Set[\theory]}}", from=1-2, to=2-2]
    \arrow[""{name=0, anchor=center, inner sep=0}, dashed, from=1-1, to=2-2]
    \arrow["\cong"{description}, draw=none, from=2-1, to=0]
    \arrow["\cong"{description}, draw=none, from=1-2, to=0]
    \end{tikzcd}\]
    we see that $\topos'$ classifies the theory $\theory_{\cong,\cong}$ whose models are triples of $\theory$-models and a pair of isomorphisms between these.
    
    However, such care will not be necessary.  Recall from \cite[Example 15]{Lack2010} that although the toposes $\Set[\theory_{\cong}]$ and $\Set[\theory_{\cong,\cong}]$ are not isomorphic as categories, they are equivalent (i.e.\ $\theory_{\cong}$ and $\theory_{\cong,\cong}$ are \emph{Morita-equivalent}).  We sidestep these issues by only working up to equivalence and referring to bipullbacks.
    
    In fact, for this specific example the iso-comma object $\Set[\theory_{\cong}]$, the pseudo-pullback $\Set[\theory_{\cong,\cong}]$ and the the (1-)pullback, which is evidently given by just $\Set[\theory]$, are all equivalent (see \cite{joyalstreet93}).  
\end{remark}

\section{Syntactic description of the representing localic groupoid}\label{sec:syndecr}

In this section we will state a presentation for a localic groupoid which represents (the classifying topos of) a geometric theory and give some intuition for the motivating ideas behind the Joyal--Tierney result.

Let $\theory$ be a (first-order) geometric theory. We would like to re-express $\theory$ in terms of \emph{propositional} geometric theories, so we can work within the simpler framework of locales instead of with the classifying topos $\Set[\theory]$. The points of this new propositional theory should somehow represent the models of the original theory $\theory$, including the objects being used to represent each sort. The question then is how to encode the sorts using a propositional theory.

\subsection{Sorts as partial equivalence relations}

If we were to focus on a single set-based model $M$, then we could include propositional variables in our language that express that $\vec{m} \in R^M$ for each relation $R$ of the theory and each appropriate tuple $\vec{m}$ of elements from $M$. More generally, we could imagine fixing some very large set $S$ and cutting out the carriers for each model as subsets of $S$. The issue is that there is generally no bound on how large the models might be.

Recall from \cref{ex:partial_surjections_from_N}, however, that the locale of partial surjections from $\N$ to any set $X$ is nontrivial. So there is a sense in which `every set is a subquotient of $\N$'. This motivates replacing the sorts in the theory $\theory$ by partial equivalence relations on $\N$, which describe these subquotients. Recall that a partial equivalence relation is a symmetric transitive relation and can be thought of as describing an equivalence relation on the subset of elements which are related to themselves. Partial equivalence relations can be axiomatised by a propositional theory.

\begin{definition}\label{def:locale_of_objects}
 For a geometric theory $\theory$, we define $G^\theory_0$ to be the classifying locale of a propositional geometric theory $P[\theory]$, defined as follows.
  \begin{itemize}
   \item For each sort $X$ of $\theory$, we add a copy of the theory of partial equivalence relations on $\N$. Explicitly, add a basic proposition ${[n \sim^X m]}$ for each $n,m \in \N$ and, for each $n,m,\ell \in \N$, and the axioms
   \begin{align*}
   {[n \sim^X m]} &\vdash {[m \sim^X n]}, && \rlap{\text{(symmetry)}} \\
   {[n \sim^X m]} \land {[m \sim^X \ell]} &\vdash {[n \sim^X \ell]}. && \rlap{\text{(transitivity)}}
   \end{align*}
   \item For each relation symbol $R \subseteq X^1 \times \dots \times X^k$ of $\theory$ (not including the equality relation), and for each $n_1,\dots,n_k \in \N$ and $m_1,\dots,m_k \in \N$, we have a basic proposition ${[(n_1,\dots,n_k) \in R]}$ and axioms
  \begin{align*}
      {[(n_1,\dots,n_k) \in R]} \land {[n_1 \sim^{X^1} m_1]} \land \dots \land {[n_k \sim^{X^k} m_k]} \vdash {[(m_1,\dots,m_k) \in R]},  \\
      {[(n_1,\dots,n_k) \in R]} \vdash {[n_1 \sim^{X^1} n_1]} \land \dots \land {[n_k \sim^{X^k} n_k]}.
   \end{align*}
  \item For each axiom $\phi \vdash_{x_1\colon X^1,\dots,x_k\colon X^k} \psi $ of $\theory$, we add an axiom 
  \[\bigwedge_{i=1}^k {[n_i \sim^{X^i} n_i]} \wedge \phi_{n_1,\dots,n_k} \vdash \psi_{n_1,\dots,n_k}\]
  for each $n_1,\dots,n_k \in \N$, where $\phi_{n_1,\dots,n_k}$ and $\psi_{n_1,\dots,n_k}$ are obtained from $\phi$ and $\psi$ by replacing each free variable $x_i$ by a (fixed) natural number $n_i$, each quantifier $\exists x\colon X.\ \chi(x,\dots)$ by a join $\bigvee_{n_x \in \N} \chi(n_x,\dots)$, each subformula of the form $(y_1,\dots,y_\ell) \in R$ with ${[(y_1,\dots,y_\ell) \in R]}$, and each subformula of the form $x \mathrel{=}_X y$ with ${[x \sim^X y]}$.
  \end{itemize}
\end{definition}
Here we have simply translated the relations on the sorts to relations on $\N$ that respect the partial equivalence relation. We have written the axioms in terms of these (with existential quantification over sorts being expressed using joins over the natural numbers).  Evidently, if $\theory$ is a propositional theory (i.e.\ there are no sorts), then $\theory$ and $P[\theory]$ are the same theory.

\begin{remark}\label{rem:points_of_G_0}
    The points of the locale $G_0^\theory$ are the models of $\theory$ where each sort is a subquotient of $\N$ (represented by a partial equivalence relation).  Given a point $p \colon 1 \to G_0^\theory$, a pair of natural numbers $n, m \in \N$ are identified by the partial equivalence relation $\sim^X$ on $\N$, corresponding to a sort $X$ of $\theory$, if and only if $p^\ast ([n \sim^X m]) = 1$, while a tuple of equivalence classes of natural numbers $([n_1], \dots , [n_k]) \in \N/{\sim^{X^1}} \times \dots \times \N/{\sim^{X^k}}$ is contained in the interpretation of the relation symbol $R$ if and only if $p^\ast([(n_1, \dots, n_k) \in R]) = 1$.
\end{remark}

\begin{remark}\label{rem:generators_for_equality_relation}
Note that the generators $[n \sim^X m]$ can be thought of as replacing the equality predicate $=_X$ on the sort $X$.  For that reason, we will often treat $[n \sim^X m]$ as a special case of $[(n,m) \in R]$ where $R$ is given by equality.
\end{remark}

\subsection{Encoding isomorphic copies}

The points of the locale $G^\theory_0$ are given by representations of models of $\theory$ as subquotients of $\N$. However, different subquotients of $\N$ might correspond to isomorphic models. To deal with this we need to construct a locale of isomorphisms.

We can write a geometric theory $\theory_{\cong}$ that describes isomorphisms between models of $\theory$ and then transform it into a propositional theory as we did for $\theory$ above.  This is precisely the theory classified by the iso-comma object described in \cref{rem:Tiso}.

\begin{definition}\label{def:locale_of_morphisms}
  We define the locale $G^\theory_1$ to be the classifying locale of a propositional geometric theory $P[\theory_{\cong}]$ (defined as above), where $\theory_{\cong}$ is a geometric theory with:
  \begin{itemize}
   \item for each sort $X$, relation symbol $R$ or axiom $\phi \vdash_{\vec{x}} \psi$ of $\theory$, two sorts $X_1,X_2$, relation symbols $R_1,R_2$ or axioms $\phi_{1,2} \vdash_{\vec{x}_{1,2}} \psi_{1,2}$ (where $R_i$ is defined on the $i$-subscripted sorts and so on),
   \item for each sort $X$ in $\theory$, a relation symbol $\alpha^X\subseteq X_1 \times X_2$ together with the axiom 
   \[(x,y) \in \alpha^X \land (x',y') \in \alpha^X \land x =_{X_1} x' \mathrel{\dashv\vdash}_{x,y\colon X_1} (x,y) \in \alpha^X \land (x',y') \in \alpha^X \land y =_{X_2} y'\]
   (where $\mathrel{\dashv\vdash}$ denotes a bidirectional sequent)
   and the axioms
   \begin{displaymath}
\begin{array}{r@{\hspace{1.5ex}}l@{\hspace{1.5ex}}l@{\quad}@{\quad}r@{}}
      & \vdash_{y\colon X_2} \exists x\colon X_1. \ (x,y) \in \alpha^X , &  \\
      & \vdash_{x\colon X_1} \exists y\colon X_2. \ (x,y) \in \alpha^X , & 
   \end{array}
   \end{displaymath}
   making $\alpha$ into the graph of a bijection\footnote{For clarity we will often write $\alpha$ suggestively as though it were a function.},
   \item for each relation symbol $R$ in $\theory$, the axioms 
   \[\bigwedge_{i=1}^k (x_i,y_i) \in \alpha^{X^i} \land (x_1,\dots,x_k) \in R_1 \mathrel{\dashv\vdash}_{x_1,\dots,x_k,y_1,\dots,y_k} \bigwedge_{i=1}^k (x_i,y_i) \in \alpha^{X^i} \land (y_1,\dots,y_k) \in R_2.\]
  \end{itemize}
\end{definition}
We remark that the third bullet point entails that corresponding basic propositions from each copy are equivalent, since they can be viewed as nullary relations. In particular, if $\theory$ is already a propositional geometric theory then $\theory$, $P[\theory]$ and $P[\theory_{\cong}]$ are all equivalent.

\begin{remark}
    Analogously to \cref{rem:points_of_G_0}, the points of $G_1^\theory$ correspond to a pair of models of $\theory$, where each sort is a subquotient of $\N$, together with an isomorphism between them.
\end{remark}

We can now form a coequaliser diagram in $\Loc$ that identifies isomorphic models in $G^\theory_0$. However, this loses information about the original theory $\theory$. (In fact, it recovers the localic reflection of $\theory$. See \cref{sec:isoclasses} below.)

The problem is that by taking the quotient we have lost information about the automorphisms of the models, and therefore about the individual (generalised) elements of each model. Indeed, the first-order theory describes objects at a higher categorical dimension than the propositional theory (propositions can only be related by implication, but there are potentially multiple different morphisms between sorts). We can retain the information by equipping $G^\theory_0$ and $G^\theory_1$ with the structure of a localic groupoid.

\begin{remark}
 One might think that we would need a localic \emph{category} recording all morphisms instead of a localic groupoid recording only the isomorphisms, but restricting to isomorphisms turns out to be sufficient.
 In a later paper we will discuss in detail how to recover the non-invertible morphisms from the localic groupoid.
\end{remark}

\begin{definition}\label{def:the_localic_groupoid}
 The localic groupoid $G^\theory$ has $G^\theory_0$ as its locale of objects, $G^\theory_1$ as is locale of morphisms and the following structure maps.
 \begin{itemize}
  \item The source map $s\colon G^\theory_1 \to G^\theory_0$ is specified by the obvious frame homomorphism defined by
  \begin{align*}
      {[(n_1,\dots,n_k) \in R]} &\mapsto {[(n_1,\dots,n_k) \in R_1]},
  \end{align*}
  where we think of the action on $[n \sim^X m]$ as the case $R = ({=_X})$ as in \cref{rem:generators_for_equality_relation}.
  \item Similarly, the target map $t\colon G^\theory_1 \to G^\theory_0$ is specified by the frame map defined by
  \begin{align*}
      {[(n_1,\dots,n_k) \in R]} &\mapsto {[(n_1,\dots,n_k) \in R_2]}.
  \end{align*}
  \item The identity map $e\colon G^\theory_0 \to G^\theory_1$ is given by frame homomorphism defined by
  \begin{align*}
      {[(n_1,\dots,n_k) \in R_1]} &\mapsto {[(n_1,\dots,n_k) \in R]}, \\
      {[(n_1,\dots,n_k) \in R_2]} &\mapsto {[(n_1,\dots,n_k) \in R]}, \\
      {[\alpha^X(n) = m]} &\mapsto {[n \sim^X m]}.
  \end{align*}
  \item The inversion map $i\colon G^\theory_1 \to G^\theory_1$ swaps the two copies of the sorts in the sense that
  \begin{align*}
      {[(n_1,\dots,n_k) \in R_1]} &\mapsto {[(n_1,\dots,n_k) \in R_2]}, \\
      {[(n_1,\dots,n_k) \in R_2]} &\mapsto {[(n_1,\dots,n_k) \in R_1]} ,\\
      {[\alpha^X(n) = m]} &\mapsto {[\alpha^X(m) = n]}.
  \end{align*}
  \item The composition map $m\colon G^\theory_1 \times_{G^\theory_0} G^\theory_1 \to G^\theory_1$ is given as follows.
  \begin{itemize}
   \item The domain of the composition map can be presented by $P[\theory_{\cong,\cong}]$ where $\theory_{\cong,\cong}$ is like $\theory_{\cong}$ above, but there are three copies of the theory $\theory$ instead of two and there are two relation symbols $\beta^X \subseteq X_1 \times X_2$ and $\gamma^X \subseteq X_2 \times X_3$ for each sort $X$, encoding two $\theory$-model isomorphisms, instead of one relation symbol $\alpha^X$.
   \item The map $m$ itself is given by the frame homomorphism for which
   \begin{align*}
      {[(n_1,\dots,n_k) \in R_1]} &\mapsto {[(n_1,\dots,n_k) \in R_1]}, \\
      {[(n_1,\dots,n_k) \in R_2]} &\mapsto {[(n_1,\dots,n_k) \in R_3]} ,\\
      {[\alpha^X(n) = p]} &\mapsto \bigvee_{m \in \N} {[\beta^X(n) = m]} \wedge {[\gamma^X(m) = p]}.
  \end{align*}
  (Intuitively, the map $m$ sends the pair of relations $(\beta^X,\gamma^X)$ to their relational composite.)
  \end{itemize}
 \end{itemize}
\end{definition}

We omit the routine proof that this is indeed a localic groupoid.

\begin{remark}
 The set $\N$ is actually only the simplest possible choice of base set for the above construction. All the properties we prove of the localic groupoid $G^\theory$ (other than those discussed in \cref{sec:countable}) will still hold if $\N$ is replaced with a larger infinite set. See \cite{lindberg} for more details.
\end{remark}

\section{An overview of the Joyal--Tierney theorem}\label{sec:joyaltierney}

We now give an overview of the Joyal--Tierney result from \cite{joyal1984galois}. In \cref{sec:proof}, we will show that the representing localic groupoid of the classifying topos $\Set[\theory]$ constructed via the Joyal--Tierney method is essentially the groupoid described above in \cref{sec:syndecr}.

This section can be summarised as follows.
\begin{itemize}
     \item In \cref{sec:descent}, we recall the theory of descent exposited in \cite{joyal1984galois}. Given a geometric morphism $f \colon \ftopos \to \topos$, this is a way to study objects of $\topos$ by equipping objects of $\ftopos$ with additional data. This data forms a topos $\Desc{f}{\ftopos}$.
     If $f \colon \ftopos \twoheadrightarrow \topos$ is an open surjection, then there is an equivalence $\Desc{f}{\ftopos} \simeq \topos$.
     \item In \cref{sec:descentandequiv} we note that $\Desc{f}{\ftopos}$ is naturally represented by a localic groupoid whenever $\ftopos$ is a localic topos. Therefore, one can obtain a representation of $\topos$ by a localic groupoid from an open surjection $\ftopos \twoheadrightarrow \topos$ whose domain is localic (called an \emph{open cover}).
     \item Finally, in \cref{sec:partequiv} we construct an open cover of every topos $\topos$ and hence conclude the Joyal--Tierney theorem that every topos is the topos of sheaves on some localic groupoid.
\end{itemize}

\subsection{Descent theory}\label{sec:descent}

In order to prove their representation theorem, Joyal and Tierney developed in \cite{joyal1984galois} a \emph{descent theory} for toposes.  We will treat descent theory as a `black  box', recalling below the necessary facts we will use in our exposition.  For details, the reader is directed to \cite[\S VIII]{joyal1984galois} and \cite[\S B1.5 and \S C5.1]{Elephant}.

Recall that if $f\colon X \to Y$ is a morphism in a (finitely complete) 1-category $\cat$, then the pullback of $f$ along itself gives the \emph{kernel pair} of $f$. This has the structure of an internal equivalence relation in $\cat$. If $f$ is a `good' quotient map (in this case, a regular epimorphism), then it can be recovered from this equivalence relation. The situation in the 2-category of toposes is similar, but instead of an internal equivalence relation, we obtain an internal groupoid.

A geometric morphism $f \colon \ftopos \to \topos$ between toposes induces an internal groupoid in $\Topos$ as in the diagram
\begin{equation*}\label{diagram:descent}
\begin{tikzcd}
    \ftopos \times_\topos \ftopos \times_\topos \ftopos \ar[shift left = 4]{r}{\pi_{2,3}} \ar{r}{\pi_{1,3}} \ar[shift right = 4]{r}{\pi_{1,2}} & \ftopos \times_\topos \ftopos \ar[shift left = 4]{r}{\pi_2} \ar[shift right = 4]{r}{\pi_1}  \arrow[loop,  "\tau "', distance=2em, in=305, out=235]
    & \ar{l}[']{\Delta} \ftopos \ar{r}{f} & \topos,
    \end{tikzcd}
\end{equation*}
where $\tau \colon  \ftopos \times_\topos \ftopos \to \ftopos \times_\topos \ftopos$ is the twist map, $\Delta \colon \ftopos \to \ftopos \times_\topos \ftopos$ is the diagonal, and the remaining maps are the appropriate projections.
\begin{definition}
    The category $\Desc{f}{\ftopos}$ of \emph{descent data} for $f$ is defined as follows.
    \begin{enumerate}
        \item The objects of $\Desc{f}{\ftopos}$ are pairs $(X,\theta)$ consisting of an object $X \in \ftopos$ and an isomorphism $\theta \colon \pi_1^\ast X \xrightarrow{\sim} \pi_2^\ast X$ of $\ftopos \times_\topos \ftopos$ such that
        \[\Delta^\ast(\theta) = \id_X \text{ and } \pi_{1,3}^\ast(\theta) = \pi_{2,3}^\ast(\theta) \circ \pi_{1,2}^\ast(\theta).\]
        This is known as a \emph{descent datum} on $X$.
        \item A morphism $g \colon (X,\theta) \to (X',\theta')$ in $\Desc{f}{\ftopos}$ is a morphism $g \colon X \to X'$ of $\ftopos$ such that the square
        \[\begin{tikzcd}
            \pi_1^\ast X \ar{r}{\theta} \ar{d}[']{\pi_1^\ast(g)} & \pi_2^\ast X \ar{d}{\pi_2^\ast(g)} \\
            \pi_1^\ast X' \ar[swap]{r}{\theta'} & \pi_2^\ast X'
        \end{tikzcd}\]
        commutes.
    \end{enumerate}
\end{definition}

The category $\Desc{f}{\ftopos}$ is a topos, and there is a \emph{canonical functor} $c^\ast \colon \topos \to \Desc{f}{\ftopos}$ that sends an object $E \in \topos$ to the pair consisting of $f^\ast E$ and the canonical isomorphism $\pi_1^\ast f^\ast E \cong \pi_2^\ast f^\ast E$ (arising from the 2-cell of the bipullback).

In fact, in \cite[\S 3]{moerdijk1988} Moerdijk shows that the topos $\Desc{f}{\ftopos}$ is obtained as the colimit in the 2-category $\Topos$ of the diagram
\[
\begin{tikzcd}
        \ftopos \times_\topos \ftopos \times_\topos \ftopos \ar[shift left = 4]{r}{\pi_{2,3}} \ar{r}{\pi_{1,3}} \ar[shift right = 4]{r}{\pi_{1,2}} & \ftopos \times_\topos \ftopos \ar[shift left = 4]{r}{\pi_2} \ar[shift right = 4]{r}{\pi_1}  \arrow[loop,  "\tau "', distance=2em, in=305, out=235]
        & \ar{l}[']{\Delta} \ftopos \ar{r} & \Desc{f}{\ftopos},
        \end{tikzcd}
\]
and the canonical functor $c^\ast \colon \topos \to \Desc{f}{\ftopos}$ is the inverse image part of the universally induced geometric morphism $\Desc{f}{\ftopos} \to \topos$.
This is analogous to how a morphism in a 1-category factors through the coequaliser of its kernel pair.

The problem of descent involves discerning for which geometric morphisms $f \colon \ftopos \to \topos$ the canonical functor $c^\ast \colon \topos \to \Desc{f}{\ftopos}$ is an equivalence. Such geometric morphisms play the same role as regular epimorphisms did in our 1-categorical analogy.

\begin{definition}
  A geometric morphism $f \colon \ftopos \to \topos$ is called an \emph{effective descent morphism} if the canonical functor $c^\ast \colon \topos \to \Desc{f}{\ftopos}$ is an equivalence.
\end{definition}

The terminology `descent' was used by Joyal and Tierney in analogy with descent theory for modules (see \cite[\S II.5]{joyal1984galois}).
If $f$ is an effective descent morphism, we say an object $(X,\theta) \in \Desc{f}{\ftopos}$ \emph{descends} along $f$ in the sense that there exists some $E \in \topos$ such that $(X,\theta) \cong (f^\ast E, \pi_1^\ast f^\ast E \cong \pi_2^\ast f^\ast E)$.

Many examples of classes of effective descent morphisms are known, including \emph{proper surjections} (see \cite[Definition C3.2.5 \& Theorem C5.1.6]{Elephant}). We will focus solely on \emph{open surjections}, which are the class of effective descent morphisms needed for the Joyal--Tierney result, and which were shown to be effective descent morphisms in \cite[Theorem VIII.2.1]{joyal1984galois}.

\subsection{Descent data with a localic domain}\label{sec:descentandequiv}

When the domain topos of a geometric morphism $f \colon \ftopos \to \topos$ is localic, say $\ftopos \simeq \Sh(G_0)$, the category of descent data $\Desc{f}{\ftopos}$ is equivalent to the topos of sheaves on some localic groupoid whose locale of objects is $G_0$. This is observed in \cite[\S VIII.3]{joyal1984galois}. To see why this is the case, we first recall two facts about localic geometric morphisms from \cref{sec:propertiesofgeomorph}. 
\begin{enumerate}
    \item Localic geometric morphisms are stable under pullback.
    \item If $f \colon \mathcal{H}' \to \mathcal{H}$ is a localic geometric morphism and $\mathcal{H}$ is a localic topos, then the topos $\mathcal{H}'$ is also localic.
\end{enumerate}
Hence, if $f \colon \ftopos \to \topos$ is a geometric morphism whose domain $\ftopos$ is a localic topos, then the pullback
\[
\begin{tikzcd}
    \ftopos \times_\topos \ftopos \ar{r}{\pi_1} \ar{d}[']{\pi_2} \arrow[dr, phantom, "\lrcorner", very near start]& \ftopos \ar{d}{f} \\
    \ftopos \ar[swap]{r}{f} & \topos
\end{tikzcd}
\]
is also a localic topos, as is the wide pullback $\ftopos \times_\topos \ftopos \times_\topos \ftopos$.  Therefore, as the fully faithful functor $\Sh \colon \Loc \to \Topos$ reflects limits, the descent diagram
\[\begin{tikzcd}
        \ftopos \times_\topos \ftopos \times_\topos \ftopos \ar[shift left = 4]{r}{\pi_{2,3}} \ar{r}{\pi_{1,3}} \ar[shift right = 4]{r}{\pi_{1,2}} & \ftopos \times_\topos \ftopos  \arrow[loop,  "\tau "', distance=2em, in=305, out=235] \ar[shift left = 4]{r}{\pi_2} \ar[shift right = 4]{r}{\pi_1} & \ar{l}[']{\Delta} \ftopos
\end{tikzcd}\]
is the image under $\Sh$ of a localic groupoid $G$:
\begin{equation}\label{diagram:grpd}
\begin{tikzcd}
 G_1 \times_{G_0} G_1 \ar[shift left = 4]{r}{\pi_{2}} \ar{r}{m} \ar[shift right = 4]{r}{\pi_{1}} & G_1 \arrow[loop,  "i"', distance=2em, in=305, out=235] \ar[shift left = 4]{r}{t} \ar[shift right = 4]{r}{s} & \ar{l}[']{e} G_0 .
\end{tikzcd}
\end{equation}

As $\ftopos \simeq \Sh(G_0)$, an object $X \in \ftopos$ is a local homeomorphism $q \colon Y \to G_0$, and descent datum $\theta \colon \pi_1^\ast(X) \to \pi_2^\ast(X)$ on $X$ is a morphism $\theta\colon s^\ast(Y) \to t^\ast (Y)$ in $\Sh(G_1)$ such that
$\id_{G_0} = e^\ast(\theta)$ and $m^\ast(\theta) = \pi_2^\ast(\theta) \circ \pi_1^\ast(\theta)$, i.e.\ an object $(Y,\theta) \in \Sh(G)$.  Similarly, arrows in $\Desc{f}{\ftopos}$ correspond to arrows in $\Sh(G)$.  Thus, there is an equivalence $\Sh(G) \simeq \Desc{f}{\ftopos}$ from which we obtain Theorem VIII.3.2 of \cite{joyal1984galois}.

\begin{theorem}\label{thm:fromeffdesctoreprgrpd}
    Let $f \colon \Sh(G_0) \to \topos$ be an effective descent morphism.  The topos $\topos$ is equivalent to the topos of equivariant sheaves on the localic groupoid $G$ whose locale of objects is $G_0$, and whose source and target maps $s, t \colon G_1 \rightrightarrows G_0$ make the square
    \[
    \begin{tikzcd}
        \Sh(G_1) \ar{r}{\Sh(s)} \ar{d}[']{\Sh(t)} & \Sh(G_0) \ar{d}{f}\\
        \Sh(G_0) \ar[swap]{r}{f} & \topos
    \end{tikzcd}
    \]
    a (bi)pullback of toposes.
\end{theorem}

Since open surjections are effective descent morphisms, this theorem applies in particular to what we call \emph{open covers}.
\begin{definition}
    An \emph{open cover} of the topos $\topos$ is an open surjection $\ftopos \twoheadrightarrow \topos$ whose domain topos $\ftopos$ is localic.
\end{definition}

Recall that open geometric morphisms are stable under (bi)pullback. So if $f \colon \Sh(G_0) \to \topos$ is an open cover, then projections $\pi_1$ and $\pi_2$ in the (bi)pullback below are open too.
\[\begin{tikzcd}
   \Sh(G_0) \times_\topos \Sh(G_0) \ar{d}[']{\pi_1} \ar{r}{\pi_2} \arrow[dr, phantom, "\lrcorner", pos = 0.05] & \Sh(G_0) \ar{d}{f} \\
    \Sh(G_0) \ar[swap]{r}{f} & \topos
\end{tikzcd}\]
This means that the source and target maps $s, t \colon G_1 \rightrightarrows G_0$ of the induced localic groupoid $G$ displayed in \cref{diagram:grpd} are open locale morphisms --- that is, $\topos$ has an \emph{open} representing groupoid.

\begin{remark}\label{rem:pullbackstable}
The same analysis holds for any other property of geometric morphisms that is stable under pullback.  For example, if the effective descent morphism $f \colon \Sh(G_0) \to \topos$ is proper or connected and locally connected, then the resulting representing groupoid for $\topos$ is also proper or connected and locally connected (in the sense that the source and target maps have these properties).
\end{remark}

\subsection{Open covers via partial equivalence relations}\label{sec:partequiv}

We are halfway to showing that every topos can be represented as the topos of sheaves on an open localic groupoid.  The remaining task is to prove that every topos has an open cover.

To find an open cover of a topos $\topos$, it suffices to find a localic geometric morphism $h \colon \topos \to \mathcal{H}$ and an open cover $f \colon \ftopos \twoheadrightarrow \mathcal{H}$, since then in the (bi)pullback
\[
\begin{tikzcd}
    \ftopos \times_\mathcal{H} \topos \ar{r}{g} \ar[two heads]{d}[']{k} \arrow[dr, phantom, "\lrcorner", very near start]& \ftopos \ar[two heads]{d}{f}  \\
    \topos \ar[swap]{r}{h} & \mathcal{H},
\end{tikzcd}
\]
the map $k \colon \ftopos \times_\mathcal{H}\topos \twoheadrightarrow \topos$ is an open surjective geometric morphism whose domain is a localic topos, as the composite $\ftopos \times_{\mathcal{H}} \topos \xrightarrow{g} \ftopos \to \Set$ is a localic morphism.
Hence, $k \colon \ftopos \times_\mathcal{H}\topos \twoheadrightarrow \topos$ is an open cover.

Suppose the topos $\topos$ classifies a theory $\theory$ with $N$ sorts. Recall from \cref{sec:classifying,sec:points} that there is a localic geometric morphism $L \colon \topos \to \Set[\NOtheory]$ which sends a $\theory$-model to the $N$ underlying objects interpreting the sorts. This will play the role of $h$ in the diagram above.

\begin{remark}
In fact, we can always choose $N$ to be $1$, since every geometric theory $\theory$ is Morita-equivalent to a single-sorted theory. This appears in \cite{joyal1984galois} as Proposition VII.3.1, but an entirely syntactic proof is given in \cite[Lemma D1.4.13]{Elephant}. In summary, the idea is to combine all the sorts of the theory into one, and introduce new unary relation symbols, $R^X$ for each sort $X$, such that $x \in R^X$ expresses the statement ``$x$ belongs to the sort $X$''. 
\end{remark}

We must now describe an open cover of $\Set[\NOtheory]$ to play the role of $f$.
As anticipated in \cref{sec:syndecr}, there is a sense in which `every set is a subquotient of $\N$' and so we are once again motivated to consider partial equivalence relations on $\N$.
Denote the classifying topos of partial equivalence relations on $N$ copies of $\N$ by $\Set[\PQtheory]$.  Explicitly, this is the propositional theory whose basic propositions are ${[n \sim^i m]}$ for each $n,m \in \N$, and $i \in N$ (meaning that $n,m$ are identified in the $i$th partial equivalence relation on $\N$), and whose axioms are
\begin{displaymath}
\begin{array}{r@{\hspace{1.5ex}}l@{\hspace{1.5ex}}l@{\quad}@{\quad}r@{}}
  {[n \sim^i m]}&\vdash& {[m \sim^i n]} & \text{(symmetry)} \\
  {[n \sim^i \ell]} \land {[\ell \sim^i m]} & \vdash  & {[n \sim^i m]} & \text{(transitivity)}
\end{array}
\end{displaymath}
for each $n,m,\ell \in \N$ and $i \in N$.

There is a geometric morphism $Q \colon \Set[\PQtheory] \to \Set[\NOtheory]$ which can be defined by its action on points as in \cref{sec:points}. It sends the $N$ partial equivalence relations on $\N$ given by a point of $\Set[\PQtheory]$ to the $N$ corresponding subquotient objects (hence giving a point of $\Set[\NOtheory]$).  This geometric morphism possesses many desirable properties: it is open and surjective, but also connected and locally connected (see \cite[Theorem C5.2.7]{Elephant}). Hence, we indeed have an open cover of $\Set[\NOtheory]$.

Now we obtain an open cover $\overline{P}_N[\topos] \twoheadrightarrow \topos$ by taking the (bi)pullback
\[\begin{tikzcd}
    \overline{P}_N[\topos] \arrow[dr, phantom, "\lrcorner", very near start] \ar{r}{Q^\ast(L)} \ar[two heads]{d}[']{L^\ast(Q)} & \Set[\PQtheory] \ar[two heads]{d}{Q} \\
    \topos \ar[swap]{r}{L} & \Set[\NOtheory].
\end{tikzcd}\]
Note that $\overline{P}_N[\topos]$ is not determined only by $\topos$, but by the map $L\colon \topos \to \Set[\NOtheory]$. This map is defined by a choice of geometric theory $\theory$ classified by $\topos$. (Recall that every topos classifies some geometric theory.)
In \cref{sec:proof} we will see that $\overline{P}_N[\topos]$ is classifying topos for the theory $P[\theory]$ defined in \cref{def:locale_of_objects} (see \cref{lem:Pfortopos_and_Pfortheory_agree}).

Finally, by applying \cref{thm:fromeffdesctoreprgrpd} we arrive at the landmark result of Joyal and Tierney \cite[Theorem VIII.3.2]{joyal1984galois}.
\begin{theorem}[Joyal--Tierney]\label{thm:joyal_tierney}
    Every Grothendieck topos can be represented as the topos of equivariant sheaves for a localic groupoid.
\end{theorem}
In \cref{sec:proof} we will see that the localic groupoid given by the above construction is precisely the one described in \cref{sec:syndecr}.
More abstractly, the theorem means that the functor $\Sh \colon \LocGrpd \to \Topos$ is essentially surjective.

\begin{remark}
 Since the geometric morphism $Q$ above is open (and even connected and locally connected), the representing localic groupoid is also open (indeed, connected and locally connected --- see \cref{rem:pullbackstable}). We will give a more hands-on proof of openness in \cref{sec:leftadjoints}.
\end{remark}

\begin{remark}\label{rem:multipleopencovers}
  A topos can have many non-equivalent open covers --- and therefore many non-isomorphic representing localic groupoids. Nonetheless, these are all equivalent in a suitable sense. See \cite[\S 7]{moerdijk1988}, though the notion of equivalence given there is weaker than necessary. We will discuss this equivalence in more detail in a later paper.
  
  The open cover $\overline{P}_N[\topos] \twoheadrightarrow \topos$ we consider is slightly different to the one built by Joyal and Tierney in \cite[Theorem VII.3.1]{joyal1984galois}. They instead use the open cover $\Set[\mathcal{TQ}_\N] \twoheadrightarrow \Set[\mathbb{O}_{>0}]$ from classifying topos of \emph{total} equivalence relations on $\N$ to the classifying topos of \emph{inhabited} objects. The reader is directed to \cite[Remark C5.2.8(c)]{Elephant} for more details.
  
  Other examples of open covers include the \emph{Diaconescu cover}, constructed in \cite{diaconescu1976} (see also \cite[Theorem C5.2.1]{Elephant} and \cite[Theorem IX.9.1]{SGL}).
\end{remark}

\section{Proof and applications of the syntactic description}\label{sec:proof} 

In this section, we prove that the localic groupoid $G^\theory$ described in \cref{sec:syndecr} is the representing localic groupoid yielded by the Joyal--Tierney method exposited in \cref{sec:joyaltierney}.  We then explore some applications of this explicit description.  In \cref{sec:leftadjoints,sec:isoclasses} we will observe that using the explicit description of $G^\theory$ we can give concrete proofs of the known facts that $G^\theory$ is an open localic groupoid and that its locale of isomorphism classes is the frame of sentences of the theory. Finally, \cref{sec:countable} we compare our localic representing groupoid with topological representing groupoids.

\subsection{Main proof}

We can now deduce our main result.  We repeat here, for the reader's convenience, the description of the localic groupoid $G^\theory$ described in \cref{sec:syndecr}.

\begin{theorem}\label{thm:main}
 Suppose $\theory$ is a geometric theory.
 Recall that the localic groupoid $G^\theory$ is defined as follows.
 \begin{itemize}
  \item The locale of objects $G^\theory_0$ is the classifying locale the propositional geometric theory $P[\theory]$, which is specified as follows.
  \begin{itemize}
  \item For each sort $X$ of $\theory$, there is a basic proposition ${[n \sim^X m]}$ for each $n,m \in \N$ together with the following axioms for each $n,m,\ell \in \N$:
  \begin{align*}
   {[n \sim^X m]} &\vdash {[m \sim^X n]}, \\
   {[n \sim^X m]} \land {[m \sim^X \ell]} &\vdash {[n \sim^X \ell]}.
   \end{align*}
  \item For each relation symbol $R \subseteq X^1 \times \dots \times X^k$ of $\theory$, and for each $n_1,\dots,n_k \in \N$ and $m_1,\dots,m_k \in \N$, we have a basic proposition ${[(n_1,\dots,n_k) \in R]}$ and axioms
  \begin{align*}
      {[(n_1,\dots,n_k) \in R]} \land {[n_1 \sim^{X^1} m_1]} \land \dots \land {[n_k \sim^{X^k} m_k]} \vdash {[(m_1,\dots,m_k) \in R]},  \\
      {[(n_1,\dots,n_k) \in R]} \vdash {[n_1 \sim^{X^1} n_1]} \land \dots \land {[n_k \sim^{X^k} n_k]}.
   \end{align*}
  \item For each axiom $\phi \vdash_{x_1\colon X^1,\dots,x_k\colon X^k} \psi $ of $\theory$, we add an axiom 
  \[\bigwedge_{i=1}^k {[n_i \sim^{X^i} n_i]} \wedge \phi_{n_1,\dots,n_k} \vdash \psi_{n_1,\dots,n_k}\]
  for each $n_1,\dots,n_k \in \N$, where $\phi_{n_1,\dots,n_k}$ and $\psi_{n_1,\dots,n_k}$ are obtained from $\phi$ and $\psi$ by replacing each free variable $x_i$ by a (fixed) natural number $n_i$, each quantifier $\exists x\colon X.\ \chi(x,\dots)$ by a join $\bigvee_{n_x \in \N} \chi(n_x,\dots)$, each subformula of the form $(y_1,\dots,y_\ell) \in R$ with ${[(y_1,\dots,y_\ell) \in R]}$, and each subformula of the form $x \mathrel{=}_X y$ with ${[x \sim^X y]}$.
  \end{itemize}
  
  \item The locale of morphisms $G^\theory_1$ is the classifying locale of a propositional geometric theory $P[\theory_{\cong}]$, where $\theory_{\cong}$ is a geometric theory with:
  \begin{itemize}
  \item for each sort $X$, relation symbol $R$ or axiom $\phi \vdash_{\vec{x}} \psi$ of $\theory$, two sorts $X_1,X_2$, relation symbols $R_1,R_2$ or axioms $\phi_{1,2} \vdash_{\vec{x}_{1,2}} \psi_{1,2}$,
  \item for each sort $X$ in $\theory$, a relation symbol $\alpha^X\subseteq X_1 \times X_2$ together with axiom 
  \[(x,y) \in \alpha^X \land (x',y') \in \alpha^X \land x =_{X_1} x' \mathrel{\dashv\vdash}_{x,y\colon X_1} (x,y) \in \alpha^X \land (x',y') \in \alpha^X \land y =_{X_2} y'\]
  and the axioms
  \begin{displaymath}
\begin{array}{r@{\hspace{1.5ex}}l@{\hspace{1.5ex}}l@{\quad}@{\quad}r@{}}
      & \vdash_{y\colon X_2} \exists x\colon X_1.\ (x,y) \in \alpha^X , &  \\
      & \vdash_{x\colon X_1} \exists y\colon X_2.\ (x,y) \in \alpha^X , & 
  \end{array}
  \end{displaymath}
  \item for each relation symbol $R$ in $\theory$, the axioms 
  \[\bigwedge_{i=1}^k (x_i,y_i) \in \alpha^{X^i} \land (x_1,\dots,x_k) \in R_1 \mathrel{\dashv\vdash}_{x_1,\dots,x_k,y_1,\dots,y_k} \bigwedge_{i=1}^k (x_i,y_i) \in \alpha^{X^i} \land (y_1,\dots,y_k) \in R_2\]
  \end{itemize}
  \item The source, target, identity and inversion maps as defined as follows:
  \begin{align*}
      s^*\colon {[(n_1,\dots,n_k) \in R]} &\mapsto {[(n_1,\dots,n_k) \in R_1]}.
  \end{align*}
  \begin{align*}
      t^*\colon {[(n_1,\dots,n_k) \in R]} &\mapsto {[(n_1,\dots,n_k) \in R_2]}.
  \end{align*}
  \begin{align*}
      e^*\colon {[(n_1,\dots,n_k) \in R_1]} &\mapsto {[(n_1,\dots,n_k) \in R]}, \\
      e^*\colon {[(n_1,\dots,n_k) \in R_2]} &\mapsto {[(n_1,\dots,n_k) \in R]}, \\
      e^*\colon {[\alpha^X(n) = m]} &\mapsto {[n \sim^X m]}.
  \end{align*}
  \begin{align*}
      i^*\colon {[(n_1,\dots,n_k) \in R_1]} &\mapsto {[(n_1,\dots,n_k) \in R_2]}, \\
      i^*\colon {[(n_1,\dots,n_k) \in R_2]} &\mapsto {[(n_1,\dots,n_k) \in R_1]} ,\\
      i^*\colon {[\alpha^X(n) = m]} &\mapsto {[\alpha^X(m) = n]}.
  \end{align*}
  \item When $G^\theory_1 \times_{G^\theory_0} G^\theory_1$ is presented by three copies of the propositions for $G^\theory_0$ together with propositions for the bijections $\beta^X \subseteq X_1 \times X_2$ and $\gamma^X \subseteq X_2 \times X_3$ for each sort, then the composition map is defined by
  \begin{align*}
      m^*\colon {[(n_1,\dots,n_k) \in R_1]} &\mapsto {[(n_1,\dots,n_k) \in R_1]}, \\
      m^*\colon {[(n_1,\dots,n_k) \in R_2]} &\mapsto {[(n_1,\dots,n_k) \in R_3]} ,\\
      m^*\colon {[\alpha^X(n) = p]} &\mapsto \bigvee_{m \in \N} {[\beta^X(n) = m]} \wedge {[\gamma^X(m) = p]}.
  \end{align*}
 \end{itemize}
 
 Then the topos of equivariant sheaves on the $G^\theory$ classifies $\theory$.
\end{theorem}

We will prove this by showing that $G^\theory$ is the groupoid obtained from the Joyal--Tierney construction we described in \cref{sec:joyaltierney}. We require one lemma before embarking on the main proof.

\begin{lemma}\label{lem:Pfortopos_and_Pfortheory_agree}
For each geometric theory $\theory$ with $N$ sorts, the commutative square
\[
\begin{tikzcd}
    \Set[P[\theory]] \ar{r} \ar{d}[']{Q'} & \Set[\PQtheory] \ar{d}{Q} \\
    \Set[\theory] \ar[swap]{r}{L} & \Set[\NOtheory]
\end{tikzcd}
\]
(where $L$ and $Q$ are defined as in \cref{sec:partequiv})
is a (bi)pullback.
\end{lemma}
\begin{proof}
For clarity we will assume the theory $\theory$ has single sort, but this is easily generalised.
Recall that $Q$ can be understood as sending the partial equivalence relation $\sim$ on $\N$ to the corresponding subquotient $\N/{\sim}$. As described in \cref{ex:limitsoftoposes}, it is easy to compute a theory $\theory'$ that the bipullback topos classifies using the methods of \cite[\S 4.5]{vickers2007locales}. We see that $\theory'$ can be taken to be the theory of pairs of a model $M$ of $\theory$, a model $\sim$ of $\PQtheorysing$ and an isomorphism $L(M) \cong Q({\sim})$. Explicitly, this means a model $M$ of $\theory$, a partial equivalence relation $\sim$ on $\N$ and a bijection $\phi\colon M[X] \to \N/{\sim}$.

It is now elementary to massage $\theory'$ into a more convenient, equivalent form by transporting all relations and functions defined on terms of $M[X]$ along the bijection $\phi$ to give ones defined in $\N/{\sim}$. Then since the sort $M[X]$ is completely specified by $\N/{\sim}$ and the bijection, it can be removed from the theory. The resulting theory is essentially propositional. We can make it manifestly propositional by replacing relations on $\N/{\sim}$ with their preimages under $\N \twoheadrightarrow \N/{\sim}$ to give subsets $U_R$ of $\N^k$, which then can described using basic generators $[(n_1,\dots,n_k) \in U_R]$ for each $(n_1,\dots,n_k) \in \N^k$.
Thus, we have arrived at the theory $P[\theory]$ described in \cref{thm:main}. This theory now has no sorts and so it is manifestly propositional.

Note that the map $Q'$ simply undoes the above translations, obtaining a quotient of $\N/{\sim}$ from the propositional theory and sending this to the single sort $X$ of $\theory$. Relations are treated in the obvious way.
\end{proof}

\begin{proof}[Proof of \cref{thm:main}]
Again we assume $\theory$ has one sort for simplicity.
Recall that the geometric morphism $L \colon \Set[\theory] \to \Set[\Otheory]$, which a model of $\theory$ to its underlying object, is localic. By \cref{lem:Pfortopos_and_Pfortheory_agree} the open cover $\overline{P}_N[\Set[\theory]] \twoheadrightarrow \Set[\theory]$ used to construct the representing groupoid in \cref{sec:partequiv} may be taken to be $Q' \colon \Set[P[\theory]] \to \Set[\theory]$.

Now by applying \cref{thm:fromeffdesctoreprgrpd}, we know $\Set[\theory]$ is represented by the localic groupoid whose locale of objects is the classifying locale of $P[\theory]$ and whose source and target maps $s,t \colon G_1^\theory \rightrightarrows G_0^\theory$ are the locale morphisms for which the square
\[
\begin{tikzcd}
    \Sh(G_1^\theory) \ar{r}{\Sh(s)} \ar{d}[']{\Sh(t)} & \Set[P[\theory]] \ar{d}{Q'} \\
    \Set[P[\theory]] \ar[swap]{r}{Q'} & \Set[\theory]
\end{tikzcd}
\]
is a bipullback of toposes. We must now show that $\Set[P[\theory_{\cong}]]$ is this bipullback.

The theory $2 \cdot \Otheory$ classifies pairs of objects, and so we deduce from \cref{ex:limitsoftoposes} that $2 \cdot \Otheory$ is classified by the product $\Set[2 \cdot \Otheory] \cong \Set[\Otheory] \times \Set[\Otheory]$.  Similarly, $\Set[\PQtheorysing] \times \Set[\PQtheorysing]$ classifies the theory $2 \cdot \PQtheorysing$ of pairs of partial equivalence relations on $\N$.  Recall also from \cref{ex:limitsoftoposes} and \cref{rem:Tiso} that the theory $\theory_{\cong}$ of isomorphisms of $\theory$-models is classified by the bipullback
\[\begin{tikzcd}
	{\Set[\theory_{\cong}]} & {\Set[\theory]} \\
	{\Set[\theory]} & {\Set[\theory]}.
	\arrow["r", from=1-1, to=1-2]
	\arrow["u"', from=1-1, to=2-1]
	\arrow["{\id_{\Set[\theory]}}", from=2-1, to=2-2, swap]
	\arrow["{\id_{\Set[\theory]}}", from=1-2, to=2-2]
	\arrow["\lrcorner"{anchor=center, pos=0.125}, draw=none, from=1-1, to=2-2]
\end{tikzcd}\]
Using the universal property of $\Set[P[\theory]]$, we find that there are induced geometric morphisms ${s,t \colon \Set[P[\theory_{\cong}]] \rightrightarrows \Set[P[\theory]]}$ such that all the squares in the diagram

\[\begin{tikzcd}[column sep = tiny, row sep = small]
	&& {\Set[P[\theory]]} && {\Set[\PQtheorysing]} \\
	\\[1.8em]
	& {\Set[P[\theory_{\cong}]]} && {\Set[2 \cdot \PQtheorysing]} \\
	&& {\Set[\theory]} && {\Set[\Otheory]} \\
	{\Set[P[\theory]]} &[3em] & {\Set[\PQtheorysing]} \\
	& {\Set[\theory_{\cong}]} && {\Set[2 \cdot \Otheory]} \\
	\\[1.8em]
	{\Set[\theory]} && {\Set[\Otheory]}
	\arrow[from=3-2, to=6-2]
	\arrow["L"{pos=0.7}, from=4-3, to=4-5]
	\arrow[from=1-3, to=1-5]
	\arrow["Q", from=1-5, to=4-5]
	\arrow["Q'"', from=1-3, to=4-3]
	\arrow[from=6-4, to=4-5]
	\arrow["r"{pos=0.7}, from=6-2, to=4-3]
	\arrow[dashed, "s", from=3-2, to=1-3]
	\arrow[dashed, "t"', from=3-2, to=5-1]
	\arrow[from=6-4, to=8-3]
	\arrow["Q'"', from=5-1, to=8-1]
	\arrow["L", from=8-1, to=8-3, swap]
	\arrow["u"', from=6-2, to=8-1]
	\arrow[from=3-4, to=1-5]
 	\arrow[from=6-2, to=6-4]
    \arrow["Q", from=5-3, to=8-3, crossing over]
    \arrow[from=5-1, to=5-3, crossing over]
    \arrow[from=3-2, to=3-4, crossing over]
    \arrow[from=3-4, to=6-4, crossing over]
    \arrow[from=3-4, to=5-3, crossing over]
    \arrow["\lrcorner"{anchor=center, pos=0.125}, draw=none, from=5-1, to=8-3]
	\arrow["\lrcorner"{anchor=center, pos=0.125}, draw=none, from=3-2, to=6-4]
	\arrow["\lrcorner"{anchor=center, pos=0.125}, draw=none, from=1-3, to=4-5]
\end{tikzcd}\]

commute up to canonical isomorphisms.
Being induced by the maps 
\[r, u \colon \Set[\theory_{\cong}] \rightrightarrows \Set[\theory],\]
which send a model a $\theory_{\cong}$-model $M \cong N$ to, respectively, $M$ and $N$, we recognise that the locale morphisms $s, t \colon G_1^\theory \rightrightarrows G_0^\theory$ corresponding to the geometric morphisms $s, t \colon \Set[P[\theory_{\cong}]] \rightrightarrows \Set[P[\theory]]$ are exactly the ones described in the hypotheses of the theorem. (Note that we are abusing notation and not differentiating between a locale morphism and its corresponding geometric morphism between localic toposes.)

Our description of the localic groupoid is therefore precisely the representing groupoid found by the Joyal--Tierney method exposited in \cref{sec:joyaltierney} provided that the square

\[
\begin{tikzcd}
    \Set[P[\theory_{\cong}]] \ar{r}{s} \ar{d}[']{t} & \Set[P[\theory]] \ar{d}{Q'} \\
    \Set[P[\theory]] \ar[swap]{r}{Q'} & \Set[\theory]
\end{tikzcd}
\]
is a bipullback of toposes.  Firstly, we note that the square commutes up to isomorphism since it can be rewritten as
\[\begin{tikzcd}
	{\Set[P[\theory_{\cong}]]} && {\Set[P[\theory]]} \\
	& {\Set[\theory_{\cong}]} & {\Set[\theory]} \\
	{\Set[P[\theory]]} & {\Set[\theory]} & {\Set[\theory]}.
	\arrow["{\id_{\Set[\theory]}}"', from=3-2, to=3-3]
	\arrow["{\id_{\Set[\theory]}}", from=2-3, to=3-3]
	\arrow["r", from=2-2, to=2-3]
	\arrow["u"', from=2-2, to=3-2]
	\arrow["\cong"{description}, draw=none, from=2-2, to=3-3]
    \arrow["\lrcorner"{anchor=center, pos=0.125}, draw=none, from=2-2, to=3-3]
	\arrow["t"', from=1-1, to=3-1]
	\arrow["s", from=1-1, to=1-3]
	\arrow[from=1-1, to=2-2]
	\arrow[from=1-3, to=2-3]
	\arrow[from=3-1, to=3-2]
	\arrow["\cong"{description}, curve={height=12pt}, draw=none, from=1-1, to=3-2]
	\arrow["\cong"{description}, curve={height=-6pt}, draw=none, from=1-1, to=2-3]
\end{tikzcd}\]
    Now for any other (bi)cone
    \[
\begin{tikzcd}
    \topos \ar{r}{f} \ar{d}[']{g} \ar[draw = none]{rd}[description]{\cong} & \Set[P[\theory]] \ar{d}{Q'} \\
    \Set[P[\theory]] \ar{r}[']{Q'} & \Set[\theory]
\end{tikzcd}
\]
of the cospan, we will demonstrate that there is a diagram of toposes and geometric morphisms
\[\begin{tikzcd}[column sep = tiny, row sep = small]
	&&& {\Set[P[\theory]]} && {\Set[\PQtheorysing]} \\
	\topos \\
	&& {\Set[P[\theory_{\cong}]]} && {\Set[2 \cdot \PQtheorysing]} \\
	&&& {\Set[\theory]} && {\Set[\Otheory]} \\
	& {\Set[P[\theory]]} &[3em] & {\Set[\PQtheorysing]} \\
	&& {\Set[\theory_{\cong}]} && {\Set[2 \cdot \Otheory]} \\
	\\[1.8em]
	& {\Set[\theory]} && {\Set[\Otheory]}
	\arrow[from=3-3, to=6-3]
	\arrow[from=6-3, to=6-5]
	\arrow[from=4-4, to=4-6]
	\arrow[from=1-4, to=1-6]
	\arrow[from=1-6, to=4-6]
	\arrow[from=1-4, to=4-4]
	\arrow[from=6-5, to=4-6]
	\arrow[from=6-3, to=4-4]
	\arrow[from=6-5, to=8-4]
	\arrow[from=5-2, to=8-2]
	\arrow[from=8-2, to=8-4]
	\arrow[from=5-4, to=8-4]
	\arrow[from=6-3, to=8-2]
	\arrow[from=3-5, to=1-6]
	\arrow["\lrcorner"{anchor=center, pos=0.125}, draw=none, from=5-2, to=8-4]
	\arrow["\lrcorner"{anchor=center, pos=0.125}, draw=none, from=3-3, to=6-5]
	\arrow["\lrcorner"{anchor=center, pos=0.125}, draw=none, from=1-4, to=4-6]
	\arrow["g"', curve={height=18pt}, from=2-1, to=5-2]
	\arrow["f", curve={height=-18pt}, from=2-1, to=1-4]
	\arrow[curve={height=8pt}, dashed, from=2-1, to=6-3]
	\arrow[dotted, from=2-1, to=3-3]
    \arrow[from=3-3, to=3-5, crossing over]
    \arrow[from=3-3, to=1-4]
    \arrow[curve={height=-18pt}, dashed, from=2-1, to=3-5, crossing over]
    \arrow[from=3-5, to=5-4, crossing over]
	\arrow[from=3-3, to=5-2, crossing over]
    \arrow[from=5-2, to=5-4, crossing over]
    \arrow[from=3-5, to=6-5, crossing over]
\end{tikzcd}\]
where every square and triangle commutes up to canonical isomorphism.
\begin{enumerate}
    \item The geometric morphism $\begin{tikzcd}[column sep = small]
         &[-20pt] \topos \ar[dashed]{r} & \Set[\theory_{\cong}] &[-20pt]
    \end{tikzcd}$ is induced by the universal property of $\Set[\theory_{\cong}]$ as in the diagram
    
    \[\begin{tikzcd}
	{\topos} && {\Set[P[\theory]]} \\
	& {\Set[\theory_{\cong}]} & {\Set[\theory]} \\
	{\Set[P[\theory]]} & {\Set[\theory]} & {\Set[\theory]}.
	\arrow["{\id_{\Set[\theory]}}"', from=3-2, to=3-3]
	\arrow["{\id_{\Set[\theory]}}", from=2-3, to=3-3]
	\arrow["r", from=2-2, to=2-3]
	\arrow["u"', from=2-2, to=3-2]
	\arrow["\cong"{description}, draw=none, from=2-2, to=3-3]
    \arrow["\lrcorner"{anchor=center, pos=0.125}, draw=none, from=2-2, to=3-3]
	\arrow["g"', from=1-1, to=3-1]
	\arrow["f", from=1-1, to=1-3]
	\arrow[dashed, from=1-1, to=2-2]
	\arrow[from=1-3, to=2-3]
	\arrow[from=3-1, to=3-2]
	\arrow["\cong"{description}, curve={height=12pt}, draw=none, from=1-1, to=3-2]
	\arrow["\cong"{description}, curve={height=-6pt}, draw=none, from=1-1, to=2-3]
\end{tikzcd}\]

    \item The geometric morphism $\begin{tikzcd}[column sep = small]
         &[-20pt] \topos \ar[dashed]{r} & \Set[2 \cdot \PQtheorysing] &[-20pt]
    \end{tikzcd}$ is universally induced by the fact that $\Set[2 \cdot \PQtheorysing] \cong \Set[\PQtheorysing] \times \Set[\PQtheorysing]$.

    \item Finally, the geometric morphism $\begin{tikzcd}[column sep = small]
         &[-20pt] \topos \ar[dotted]{r} & \Set[P[\theory_{\cong}]] &[-20pt]
    \end{tikzcd}$ is induced by the universal property of $\Set[P[\theory_{\cong}]]$ as in the diagram
    \[\begin{tikzcd}
	\topos \\
	& {\Set[P[\theory_{\cong}]]} & {\Set[2 \cdot \PQtheorysing]} \\
	& {\Set[\theory_{\cong}]} & {\Set[2 \cdot \Otheory]}.
	\arrow[from=2-2, to=3-2]
	\arrow[from=2-2, to=2-3]
	\arrow[from=2-3, to=3-3]
	\arrow[from=3-2, to=3-3]
	\arrow[dotted, from=1-1, to=2-2]
	\arrow[""{name=0, anchor=center, inner sep=0}, dashed, curve={height=-24pt}, from=1-1, to=2-3]
	\arrow[""{name=1, anchor=center, inner sep=0}, dashed, curve={height=24pt}, from=1-1, to=3-2]
	\arrow["\cong"{description}, draw=none, from=2-2, to=3-3]
    \arrow["\lrcorner"{anchor=center, pos=0.125}, draw=none, from=2-2, to=3-3]
	\arrow["\cong"{description}, draw=none, from=2-2, to=0]
	\arrow["\cong"{description, pos=0.2}, draw=none, from=2-2, to=1]
\end{tikzcd}\]
\end{enumerate}

Thus, the (bi)cone factorises canonically as
\[\begin{tikzcd}
	\topos \\
	& {\Set[P[\theory_{\cong}]]} & {\Set[P[\theory]]} \\
	& {\Set[P[\theory]]} & {\Set[\theory]}.
	\arrow["t"', from=2-2, to=3-2]
	\arrow["s", from=2-2, to=2-3]
	\arrow["{Q'}", from=2-3, to=3-3]
	\arrow["{Q'}"', from=3-2, to=3-3]
	\arrow["\cong"{description}, draw=none, from=2-2, to=3-3]
	\arrow[dotted, from=1-1, to=2-2]
	\arrow[""{name=0, anchor=center, inner sep=0}, "f", curve={height=-24pt}, from=1-1, to=2-3]
	\arrow[""{name=1, anchor=center, inner sep=0}, "g"', curve={height=24pt}, from=1-1, to=3-2]
	\arrow["\cong"{description}, draw=none, from=2-2, to=0]
	\arrow["\cong"{description, pos=0.2}, draw=none, from=2-2, to=1]
\end{tikzcd}\]
We have elided the details that $\Set[P[\theory_{\cong}]]$ also satisfies the necessary universal property on 2-cells to be the bipullback, but this can be demonstrated in a similar fashion since the canonical morphism $\begin{tikzcd}[column sep = small]
         &[-20pt] \topos \ar[dotted]{r} & \Set[P[\theory_{\cong}]] &[-20pt]
    \end{tikzcd}$ was universally induced by a series of bilimits.

Finally, by demonstrating in an analogous manner that $\Set[P[\theory_{\cong,\cong}]]$ is equivalent to the wide bipullback $\Set[P[\theory]] \times_{\Set[\theory]} \Set[P[\theory]] \times_{\Set[\theory]} \Set[P[\theory]]$, we recognise that the composition map of our groupoid is described as in the hypotheses, thus completing the proof that the localic groupoid $G^\theory$ represents $\Set[\theory]$.
\end{proof}

\begin{example}
    As remarked below \cref{def:locale_of_objects}, when $\theory$ is a propositional theory, the theories $\theory$, $P[\theory]$ and $P[\theory_{\cong}]$ are all equivalent, and therefore have isomorphic classifying locales.  Hence, the syntactic groupoid $G^\theory$ as described in \cref{thm:main} is an example of a categorically discrete localic groupoid in the sense of \cref{ex:grpdexs}(\ref*{example1}) and so, as in \cref{exsrevisited}, the topos of equivariant sheaves $\Sh(G^\theory)$ is equivalent to the topos of sheaves on the classifying locale of $\theory$, as we would expect.
\end{example}

\subsection{Explicit left adjoints}\label{sec:leftadjoints}

We noted in \cref{sec:joyaltierney} that $G^\theory$ is an \emph{open} localic groupoid by general properties of the Joyal--Tierney construction. However, it is instructive to also see this directly in terms of an explicit left adjoint map.

\begin{lemma}\label{lem:adjoint_to_s}
 The frame map $s^*$ corresponding to the source morphism $s\colon G^\theory_1 \to G^\theory_0$ of $G^\theory$ has a left adjoint $s_!\colon \O G^\theory_1 \to \O G^\theory_0$ defined by
 \begin{equation}\label{eq:leftadjoint}
     \begin{aligned}
  s_! &\left(\bigwedge_{i \in I} [\vec{a}^i \in R^i_1] \wedge \bigwedge_{j \in J} [\vec{b}^j \in R^j_2] \wedge \bigwedge_{k \in K} [\alpha^{X^k}(c^k) = d^k] \right) \\ &{} = \bigwedge_{i \in I} [\vec{a}^i \in R^i] \wedge \bigvee_{\vec{y} \in \N^V} \bigwedge_{j \in J} [\vec{\pi}^j(\vec{y}) \in R^j] \wedge \bigwedge_{k \in K} [c^k \sim^{X^k} \pi^k(\vec{y})].
 \end{aligned}
 \end{equation}
 Here $V$ is the set of `distinct variables' represented by the $b^j_\ell$ or $d^k$ values. Explicitly, elements of $V$ are pairs $(n,X)$ where $n$ is natural number chosen from 
 \[\bigcup_{j \in J} \{b^j_1,\dots,b^j_\ell\} \cup \{d^k \mid k \in K\}\]
 and $X$ is the sort corresponding to the type of the variable in question.  The maps $\vec{\pi}^j$ and $\pi^k$ simply project out the values indexed by the appropriate variables.
\end{lemma}
\begin{remark}
  The left adjoint of an open frame homomorphism is a pointfree incarnation of the direct image map of an open continuous function.
  It is then not too surprising that $s_!$ is related to existential quantification, since an object should intuitively lie in the image of $u \in \O G^\theory_1$ under the continuous map $s$ if there exists a morphism in $u$ which maps to it.
  From a logical perspective, $s_!$ sends conjunctions of logical formulae involving variables from both the domain and codomain sorts to formulae involving only variables from the domain sorts by existentially quantifying over those variables in the codomain sorts (here implemented as a join over $\N$).
\end{remark}

The bookkeeping necessary to define the left adjoint can obfuscate the core idea. To make this clearer we give a number of examples using the theory of total inhabited orders (see \cref{ex:linear_order}).
Recall that natural numbers encode (arbitrary) values of variables and do not represent their own values. Also note that in many-sorted theories the same natural number may encode different variables as long as the sorts of the variables differ. This is why the set $V$ defined above involves both the value and the type of each index.
\begin{itemize}
  \item Variables from domain sorts are left alone: ${[1 \mathrel{\le}_1 2]} \mapsto {[1 \le 2]}$.
  \item Variables from codomain sorts are `projected out': ${[1 \mathrel{\le}_2 2]} \mapsto \bigvee_{y_1,y_2 \in \N} {[y_1 \le y_2]}$.
  \item This also happens for isomorphisms: ${[\alpha(1) = 2]} \mapsto \bigvee_{x' \in \N} {[1 \sim x']}$.
  \item Different variables are quantified over independently:
  \begin{align*}
      {[1 \mathrel{\le}_2 2]} \wedge {[\alpha(1) = 4]} & \mapsto ({\textstyle\bigvee}_{y_1,y_2 \in \N} {[y_1 \le y_2]}) \wedge ({\textstyle\bigvee}_{y_4 \in \N} {[1 \sim y_4]}) \\
      &\, = {\textstyle\bigvee}_{y_1,y_2,y_4 \in \N} [y_1 \le y_2] \wedge [1 \sim y_4].
  \end{align*}
  \item Different instances of the same variable vary in lockstep: 
  \[\textstyle {[1 \mathrel{\le}_2 2]} \wedge {[\alpha(1) = 1]} \mapsto \bigvee_{y_1,y_2 \in \N} {[y_1 \le y_2]} \wedge {[1 \sim y_1]}.\]
\end{itemize}

\begin{proof}[Proof of \cref{lem:adjoint_to_s}]
As a left adjoint, the map $s_!$ preserves arbitrary joins, so it is completely determined by where it sends basic opens in $\O G^\theory_1$. These basic opens are given by finite meets of generators.
To avoid confusion we temporarily refer the suplattice map defined by \cref{eq:leftadjoint} on basic opens as $h\colon \O G^\theory_1 \to \O G^\theory_0$.

Recall that suplattices are complete join-semilattices and their homomorphisms are join-preserving maps.  Evidently, every frame is an example of a suplattice.

To see that this definition of $h$ indeed gives a suplattice homomorphism, we use the coverage theorem (see \cite[\S 5.2]  {abramsky1993quantales}), which asserts that, given a frame presentation with relations $\bigvee_\alpha \bigwedge_r S^r_\alpha \le \bigvee_\alpha \bigwedge_r T^r_\alpha$, there is a suplattice presentation for the underlying suplattice of the frame where the generators are formal finite meets of the frame generators and the relations are given by $\bigvee_\alpha (g \wedge \bigwedge_r S^r_\alpha) \le \bigvee_\alpha (g \wedge \bigwedge_r T^r_\alpha)$ for each suplattice generator $g$.

So to prove $h$ is well-defined we must show that, for every relation in the frame representation of $\O G^\theory_1$, the image of the corresponding relation obtained by taking a meet with finite meets of generators becomes an inequality in $\O G^\theory_0$.

\begin{itemize}
  \item For the relations involving only the basic relations $R_1^i$ from the domain copy this is immediate.
  \item Now we consider the relations only involving the codomain relations $R_2^j$ (including $\sim^{X_2}$ relations).
  \begin{itemize}
    \item The symmetry axiom ${[n \sim^{X_2} m]} \le {[m \sim^{X_2} n]}$ is easily seen to be preserved since a similar symmetry axiom holds for $\sim^X$ in $\O G^\theory_0$.
    \item The transitivity axiom for $\sim^{X_2}$ gives the relation 
          \[g \wedge {[n \sim^{X_2} m]} \wedge {[m \sim^{X_2} \ell]} \le g \wedge {[n \sim^{X_2} \ell]}.\]
          Applying the putative $h$ map to both sides we see that there is potentially an extra variable $m$ on the left-hand side. However, for every $m$ the transitivity axiom for $\sim^X$ in $\O G^\theory_0$ gives the desired inequality, and so by taking joins over all $m \in \N$ we conclude that $h$ preserves the symmetry axiom.
    \item We can then handle the other axioms involving $R_2^i$ relations in a very similar way.
  \end{itemize}
  \item Finally, we consider the axioms involving $\alpha^{X^k}$. These are proved in a similar way to above, but instead of using analogous axioms in $\O G^\theory_0$ to prove the inequalities, we use properties of $\sim^{X^k}$.
  \begin{itemize}
      \item Functionality and injectivity of $\alpha^X$ (the first axiom of $\alpha^X$ in \cref{thm:main}) can be reduced to transitivity (and symmetry) of $\sim^X$ in $\O G^\theory_0$.
      \item The claim for the relation 
            \[g \wedge {[\alpha^{X^k}(c) = d]} \le g \wedge {[c \sim^{X^k} c]} \wedge {[d \sim^{X^k} d]}\]
            also follows from transitivity and symmetry, as these give 
            \[{[c \sim^{X^k} y]} \le {[c \sim^{X^k} c]} \wedge {[y \sim^{X^k} y]}.\]
      \item Compatibility of $\alpha$ with other relations reduces to compatibility of these relations with $\sim$.
      \item It only remains to consider the surjectivity and totality axioms, which state $g \wedge {[y \sim^X y]} \le \bigvee_{x \in \N} g \wedge {[\alpha^X(x) = y]}$ and $g \wedge {[x \sim^X x]} \le \bigvee_{y \in \N} g \wedge {[\alpha^X(x) = y]}$, respectively. After applying $h$ we have valid inequalities since in the first case we can take $x = y$ in the join and in the second case we can take $y = x$ in the join.
  \end{itemize}
\end{itemize}

Thus, $h$ is a well-defined suplattice homomorphism. We now show that it really is the left adjoint to $s^*$. It is clear that $hs^* = \id_{\O G^\theory_0}$. We must prove that $s^*h \ge \id_{\O G^\theory_1}$. It suffices to show this on basic opens.

Let $g = \bigwedge_{i \in I} [\vec{a}^i \in R^i_1] \wedge \bigwedge_{j \in J} [\vec{b}^j \in R^j_2] \wedge \bigwedge_{k \in K} [\alpha^{X^k}(c^k) = d^k]$ be such a basic open. We will employ the shorthand ${[\vec{b}^j \sim \vec{b}^j]} = {[b^j_1 \sim b^j_1]} \wedge \dots \wedge {[b^j_\ell \sim^{X^j_2} b^j_\ell]}$.  Using the relations on the generators of $G_1^\theory$, we can show $\bigwedge_{j \in J} [\vec{b}^j \in R^j_2] \le [\vec{b}^j \sim \vec{b}^j]$. Similarly, we have that ${[\alpha^{X^k}(c^k) = d^k]} \le {[d^k \sim^{X^k_2} d^k]}$.

Then $[b \sim^{X_2} b] = \bigvee_{y_b \in \N} [\alpha(y_b) = b]$ and so (by grouping the joins over $y_b$ for the $b$'s corresponding to the same variables) we find that
\begin{align*}
  g &\le \bigwedge_{i \in I} [\vec{a}^i \in R^i_1] \wedge \bigvee_{\vec{y} \in \N^V} \bigwedge_{j \in J} ([\vec{b}^j \in R^j_2] \wedge [\vec{\alpha}(\vec{\pi}^j(\vec{y})) = \vec{b}^j]) \\ &\qquad {} \wedge \bigwedge_{k \in K} ([\alpha^{X^k}(c^k) \sim^{X^k_2} d^k] \wedge [\alpha^{X^k}(\pi^k(\vec{y})) = d^k] \wedge [d^k \sim^{X^k_2} d^k]).
\end{align*}
Now note that $[\vec{b}^j \in R^j_2] \wedge [\vec{\alpha}(\vec{\pi}^j(\vec{y})) = \vec{b}^j] \le [\vec{\pi}^j(\vec{y}) \in R^j_1]$ by the compatibility of $\alpha$ and $R^j$ and that $[\alpha^{X^k}(c^k) \sim^{X^k_2} d^k] \wedge [\alpha^{X^k}(\pi^k(\vec{y})) = d^k] \wedge [d^k \sim^{X^k_2} d^k] \le [c^k \sim^{X^k_1} \pi^k(\vec{y})]$ by injectivity of $\alpha$.
So we obtain that
\begin{align*}
g &\le \bigwedge_{i \in I} [\vec{a}^i \in R^i_1] \wedge \bigvee_{\vec{y} \in \N^V} \bigwedge_{j \in J} [\vec{\pi}^j(\vec{y}) \in R^j_1] \wedge \bigwedge_{k \in K} [c^k \sim^{X^k_1} \pi^k(\vec{y})].
\end{align*}
But the right-hand side of this inequality is precisely $s^*h(g)$ and hence we are done.
\end{proof}

With the explicit description of the left adjoint in hand, showing that the Frobenius reciprocity condition is satisfied is now trivial.  Explicitly, we have the following equalities for basic opens of $\O G_0^\theory$ and $\O G_1^\theory$.
\[
\begin{aligned}
  s_! &\left( s^\ast \left(\bigwedge_{\ell \in L} [\vec{e}^{\, \ell} \in R^{\ell}] \right) \wedge \bigwedge_{i \in I} [\vec{a}^i \in R^i_1] \wedge \bigwedge_{j \in J} [\vec{b}^j \in R^j_2] \wedge \bigwedge_{k \in K} [\alpha^{X^k}(c^k) = d^k] \right) \\
  & = s_!\left(\bigwedge_{\ell \in L} [\vec{e}^{\, \ell} \in R_1^\ell] \wedge \bigwedge_{i \in I} [\vec{a}^i \in R^i_1] \wedge \bigwedge_{j \in J} [\vec{b}^j \in R^j_2] \wedge \bigwedge_{k \in K} [\alpha^{X^k}(c^k) = d^k] \right) \\
  &{} = \bigwedge_{\ell \in L} [\vec{e}^{\, \ell} \in R^\ell] \wedge  \bigwedge_{i \in I} [\vec{a}^i \in R^i] \wedge \bigvee_{\vec{y} \in \N^V} \bigwedge_{j \in J} [\vec{\pi}^j(\vec{y}) \in R^j] \wedge \bigwedge_{k \in K} [c^k \sim^{X^k} \pi^k(\vec{y})] \\
  & = \bigwedge_{\ell \in L} [\vec{e}^{\, \ell} \in R^l] \wedge s_!\left(\bigwedge_{i \in I} [\vec{a}^i \in R^i_1] \wedge \bigwedge_{j \in J} [\vec{b}^j \in R^j_2] \wedge \bigwedge_{k \in K} [\alpha^{X^k}(c^k) = d^k]\right)
 \end{aligned}
\]

Therefore, the source map is open. Hence, we have given another proof for the following result.

\begin{proposition}
  The representing localic groupoid $G^\theory$ is an open groupoid.
\end{proposition}

\subsection{The isomorphism classes}\label{sec:isoclasses}

Recall that descent theory for toposes expresses the topos $\Set[\theory]$ as a colimit of groupoid diagram obtained by taking sheaves on the representing groupoid $G^\theory$.
Now, as a left adjoint, the localic reflection preserves colimits, and so the colimit of the diagram for $G^\theory$ in $\Loc$ gives the \emph{localic reflection} of $\Set[\theory]$. Since parallel 2-cells in $\Loc$ are equal, this colimit can be replaced by the coequaliser of the source and target maps.

Let $\pi_0(G^\theory)$ be the coequaliser
\[
\begin{tikzcd}
    G_1^\theory \ar[shift left]{r}{t} \ar[shift right]{r}[']{s} & G_0^\theory \ar[two heads]{r} & \pi_0(G^\theory)
\end{tikzcd}
\]
in $\Loc$, which we call the \emph{locale of isomorphism classes} of $G^\theory$.  Indeed, if $G^\theory$ is a spatial groupoid (see \cref{sec:countable} below) then $\pi_0(G^\theory)$ is the locale associated to the space of isomorphism classes of objects in $G_0^\theory$, endowed with the quotient topology.

We have shown that $\pi_0(G^\theory)$ is the localic reflection of $\Set[\theory]$ (see \cite[Lemma C5.3.7]{Elephant}). So $\O \pi_0(G^\theory)$ is isomorphic to the frame of subterminals in $\Set[\theory]$. The frame of subterminals of the classifying topos $\Set[\theory]$ is known to be the frame of \emph{sentences} of the theory, i.e.\ the frame whose opens are $\theory$-provable equivalence classes of formulae without free variables ordered by $\theory$-provability.

We can also obtain the localic reflection of $\Set[\theory]$ with our approach. By \cite[Proposition 1.3]{kock1989godement} the opens of the locale of isomorphism classes of an open localic groupoid are in bijection with the fixed points of the closure operator $s_!t^*$.
Note that in our case $s_!t^*$ sends $\bigwedge_{j \in J} [\vec{b}^j \in R^j]$ to $\bigvee_{\vec{y} \in \N^V} \bigwedge_{j \in J} [\vec{\pi}^j(\vec{y}) \in R^j]$.
A general element of $\O G^\theory_0$ is given by joins of the generators, which correspond to quantifier-free formulae in $\theory$, but with variables replaced by certain natural numbers. The order in the frame is given by provable entailment. The closure operator $s_! t^\ast$ takes a join over all possible natural numbers, which has the effect of existentially quantifying over the free variables. Thus, the fixed points of $s_! t^\ast$ then correspond to equivalence classes of formulae of $\theory$ with no free variables, as required.

It is also possible to use the results of \cite{manuell2022} to find an explicit presentation of the quotient locale $\pi_0(G^\theory)$ by generators and relations.

\subsection{The case of countable theories and topological groupoids}\label{sec:countable}

Under certain countability restrictions on the theory $\theory$, our construction can be understood to give a \emph{topological groupoid}. Some readers might find this preferable to working with locales.

\begin{definition}
 We say a geometric theory is \emph{countable} if it has a countable number of sorts, relations and axioms.
\end{definition}

\begin{proposition}
 For a countable geometric theory $\theory$, the localic groupoid constructed in \cref{thm:main} is spatial and thus arises from a topological groupoid.
\end{proposition}
\begin{proof}
 Note that if a theory $\theory$ is countable, then the locale of objects and the locale of morphisms of the representing localic groupoid are countably presented.
 
 Assuming excluded middle, a countably presented locale is spatial (see \cite{heckmann2015}). Moreover, since countably presented locales are closed under finite limits, the domain of the composition map is also spatial, as required.
\end{proof}
\begin{remark}
 Equivariant sheaves on a topological groupoid are defined analogously to those on a localic groupoid. It is perhaps not obvious that the concepts coincide when they both apply, since pullbacks of products of locales and spaces might differ. Nonetheless they do agree, since if $X \to G_0$ is a local homeomorphism and $G_0$ is spatial, then so is $X$ (see \cite[Lemma C1.3.2]{Elephant}), and since local homeomorphisms are stable under pullback.
\end{remark}

When $\theory$ is a countable geometric theory, the topological groupoid obtained from \cref{thm:main} is the same representing topological groupoid as constructed by Forssell in \cite{forssell2012}, which we now recall.  

Let $\theory$ be a geometric theory with a conservative set of $\Set$-based models, and let $\mathcal{S}$ be an infinite `indexing' set for these models.  A model $M$ of $\theory$ is said to be $\mathcal{S}$-indexed if the underlying set of each sort is a subquotient of $\mathcal{S}$.  For a tuple $\vec{a} \in \mathcal{S}$, we will write $[\vec{a}]$ for its equivalence class in the subquotient.
Forssell shows in \cite[Theorem 5.1]{forssell2012} that the classifying topos $\Set[\theory]$ is equivalent to the topos of equivariant sheaves on the topological groupoid
\begin{equation}\label{diag:topgrpd}
\begin{tikzcd}
    I_\theory^\mathcal{S} \times_{M_\theory^\mathcal{S}} I_\theory^\mathcal{S} \ar{r}{m} &
    I_\theory^\mathcal{S} \ar[shift left = 3]{r}{t} \ar[shift right = 3]{r}{s} &
    M_\theory^\mathcal{S} \ar{l}[']{e}
    \end{tikzcd}
\end{equation}
constructed as follows.
    \begin{enumerate}
    \item The \emph{space of $\mathcal{S}$-indexed models} $M^\mathcal{S}_\theory$ is the set of all $\mathcal{S}$-indexed models of $\theory$ endowed with the \emph{logical topology for objects} --- the topology generated by subsets of form
    \[\class{\vec{a} \in R}_{M^\mathcal{S}_\theory} = \{\,M \in M^\mathcal{S}_\theory \mid [\vec{a}] \in R^{M}\,\},\]
    where $R$ is a relation of $\theory$ (including equality), $R^M$ is its interpretation in a model $M$, and $\vec{a}$ is a tuple of elements of $\mathcal{S}$.
    \item The space of arrows $I^\mathcal{S}_\theory$ is the set of all isomorphisms between models in $M^\mathcal{S}_\theory$
    endowed with the \emph{logical topology for arrows} --- the topology generated by sets of the form
    \[\left\llbracket\begin{matrix}
    \vec{a} \in R \\
    \vec{b} \mapsto \vec{c} \\
    \vec{d} \in R'
    \end{matrix}
    \right\rrbracket_{I_\theory^\mathcal{S}}
    = \left\{\begin{array}{c|c}
         &  [\vec{a}] \in R^{M}, \\
        M \xrightarrow{\alpha} M' \in I_\theory^\mathcal{S} & [\vec{b}] \in M, \, [\vec{c}] \in M', \,  \alpha([\vec{b}]) = [\vec{c}], \\
         & [\vec{d}] \in {R'}^{{M'}}
    \end{array}\right\}\]
    \item The maps $m$, $t$, $e$, $s$ and $i$ are defined in the obvious way.
\end{enumerate}

By \cite[Theorem 6.2.4]{makkaireyes} a countable theory $\theory$ has a conservative set of $\Set$-based models, and then by the downward L\"{o}wenheim-Skolem theorem these models can be taken to be countable.  Thus, $\theory$ has enough $\N$-indexed models.  We immediately recognise the locales $G_0^\theory$ and $G_1^\theory$ constructed in \cref{def:locale_of_objects} and \cref{def:locale_of_morphisms} as the locales of opens for, respectively, the logical topology for objects and arrows on the sets $M_\theory^{\N}$ and $I_\theory^{\N}$.  Explicitly, for each $\vec{n},\vec{n}',\vec{m},\vec{m}' \in \N$, we identify the basic open $\class{\vec{n} \in R}_{M^\N_\theory} \subseteq M^\N_\theory$ with the generator $[\vec{n} \in R]$ of $G_0^\theory$,
and similarly the basic open
\[
\left\llbracket\begin{matrix}
    \vec{n} \in R \\
    \vec{m} \mapsto \vec{m}' \\
    \vec{n}' \in R'
    \end{matrix}
    \right\rrbracket_{I_\theory^\N} \subseteq I^\N_\theory
    \text{ is identified with }
    [\vec{n} \in R_1] \land [\vec{n}' \in R'_2] \land \bigwedge_{m_i \in \vec{m}} [\alpha(m_i) = m'_i].
\]
Thus, when the theory $\theory$ is countable, the localic groupoid constructed in \cref{thm:main} coincides with the topological groupoid of $\N$-indexed models in \cref{diag:topgrpd}.

\begin{remark}
For a countable theory $\theory$, the representing topological groupoid for $\Set[\theory]$ constructed by Butz and Moerdijk in \cite{butz1998representing} is not directly comparable with the groupoid we build in \cref{thm:main}, instead deriving from one of the many other open covers of $\Set[\theory]$.  In summary, it is the groupoid obtained when, instead of considering the theory $\PQtheorysing$ of partial equivalence relations on $\N$ as we did, one takes the theory of partial equivalence relations on $\N$ where every equivalence class is infinite --- that is, the theory obtained by adding to $\PQtheorysing$, for each $n, \ell \in \N$, the axiom
\[[n \sim n] \vdash \bigvee \{[n \sim m_1] \land \dots \land [n \sim m_\ell]  \mid m_i \in \N \text{ with } m_1 < m_2 < \dots < m_\ell \}.\]
\end{remark}

\appendix
\section{Descent data and equivariant sheaves}\label{app:equivanddesc}

In this appendix we explicitly spell out the equivalence between the datum of a compatible $G_1$-action on a local homeomorphism $q \colon Y \to G_0$ and descent datum $(Y,\theta)$ for a localic groupoid $G$.  Thereby, we are free to using either definitions when discussing the topos of sheaves $\Sh(G)$.  The equivalence is merely a case of unravelling definitions, but since this can at times be fiddly, we include an exposition here.

Given a local homeomorphism $q \colon Y \to G_0$ with a compatible $G_1$-action $\beta \colon Y \times_{G_0} G_1 \to Y$, the corresponding descent datum is the pair $(Y,\theta_\beta)$ where $\theta_\beta$ is the induced map
\[
\begin{tikzcd}
    s^\ast(Y) \ar[bend right]{rdd} \ar[bend left]{rrd}{\beta} \ar[dashed]{rd}{\theta_\beta} && \\
    & t^\ast(Y) \arrow[dr, phantom, "\lrcorner", very near start] \ar{r} \ar{d} & Y \ar{d}{q} \\
    & G_1 \ar[swap]{r}{t} & G_0,
\end{tikzcd}
\]
where the outside square commutes by the axiom $q(\beta(y,g)) = t(g)$ of $\beta$.

In `point-set' notation, the locales $s^\ast(Y)$ and $t^\ast(Y)$ are the spaces
\begin{align*}
    s^\ast(Y)  & = \{(y,f) \in Y \times G_1\mid s(f) = q(y)\}, \\
    t^\ast(Y) & = \{(y,f) \in Y \times G_1 \mid t(f) = q(y)\},
\end{align*}
and $\theta_\beta$ is the map which sends $(y,f) \in s^\ast(Y)$ to $(\beta(y,f),f) \in t^\ast(Y)$.  We first show that $\theta_\beta$ does indeed define descent datum on $Y$.

The condition $e^\ast(\theta_\beta) = \id_Y$ asserts that the map $e^\ast(\theta_\beta)$ in the composite pullback diagram below is canonically the identity on $Y$.
\[
\begin{tikzcd}
    e^\ast s^\ast (Y) \ar{r} \ar{d}[']{e^\ast(\theta_\beta)} \arrow[dr, phantom, "\lrcorner", very near start] & s^\ast(Y) \ar{d}{\theta_\beta} \\
    e^\ast t^\ast (Y) \ar{r} \ar{d} \arrow[dr, phantom, "\lrcorner", very near start] & t^\ast(Y) \ar{d} \\
    G_0 \ar[swap]{r}{e} & G_1
\end{tikzcd}
\]

The space $e^\ast s^\ast (Y)$ is given by
\begin{align*}
  e^\ast s^\ast (Y) &= \{(x,y,f) \in G_0 \times Y \times G_1 \mid e(x) = f, \ s(f) = q(y)\} \\
\intertext{and similarly}
 e^\ast t^\ast (Y) &= \{(x,y,f) \in G_0 \times Y \times G_1 \mid e(x) = f, \ t(f) = q(y) \}.
\end{align*}
The map $e^\ast (\theta_\beta) \colon e^\ast s^\ast (Y) \to e^\ast t^\ast (Y)$ acts by
\[(x,y,f) \mapsto (x,\beta(y,f),f).\]
But since $x = s(e(x)) = s(f) = q(y)$, a triple $(x,y,f) \in e^\ast s^\ast (Y)$ is entirely determined by $y$.  Thus, there is a canonical isomorphism $e^\ast s^\ast (Y)\cong Y$ given by projecting onto the second component of the tuple.  Similarly, the same projection yields an isomorphism $ e^\ast t^\ast (Y) \cong Y$.  Since $f = e(q(y))$ for each $(x,y,f) \in e^\ast s^\ast (Y)$, we observe that $\beta(y,f) = \beta(y,e(q(y))) = y$.  Thus, we have a commuting triangle
\[
\begin{tikzcd}[row sep = small, column sep = small]
    & e^\ast s^\ast(Y) \ar{dd}{e^\ast(\theta_\beta)} \ar[sloped]{ld}{\sim}\\
    Y   & \\
    & \ar[sloped]{lu}{\sim} e^\ast t^\ast (Y),
\end{tikzcd}
\]
as required.

Now we turn to the condition that $m^\ast(\theta_\beta) = \pi_2^\ast(\theta_\beta) \circ \pi_1^\ast(\theta_\beta)$.  The spaces involved can be expressed as
\begin{alignat*}{6}
    \pi_1^\ast s^\ast(Y) &=  \{(y,f,g) \in Y \times G_1 \times G_1 && {} \mid s(\pi_1(f,g)) &&= s(f) &&= q(y), \ && t(f) = s(g) \},  \\
    \pi_1^\ast t^\ast(Y) &=  \{(y,f,g) \in Y \times G_1 \times G_1 && {} \mid t(\pi_1(f,g)) &&= t(f) &&= q(y), \ && t(f) = s(g) \}, \\
    \pi_2^\ast s^\ast(Y) &=  \{(y,f,g) \in Y \times G_1 \times G_1 && {} \mid s(\pi_2(f,g)) &&= s(g) &&= q(y), \ && t(f) = s(g) \}, \\
    \pi_2^\ast t^\ast(Y) &=  \{(y,f,g) \in Y \times G_1 \times G_1 && {} \mid t(\pi_1(f,g)) &&= t(g) &&= q(y), \ && t(f) = s(g) \}.
\end{alignat*}
Using the equations $s \circ m = s \circ \pi_1$ and $t \circ m = t \circ \pi_2$, and the commutativity of the pullback square
\[
\begin{tikzcd}
    G_1 \times_{G_0} G_1 \ar{r}{\pi_2} \ar{d}[']{\pi_1} \arrow[dr, phantom, "\lrcorner", very near start] & G_1 \ar{d}{s} \\
    G_1 \ar[swap]{r}{t}  & G_0,
\end{tikzcd}
\]
we conclude that $m^\ast s^\ast(Y) = \pi_1^\ast s^\ast(Y)$, $m^\ast t^\ast(Y) = \pi_2^\ast t^\ast(Y)$ and $\pi_1^\ast t^\ast(Y)= \pi_2^\ast s^\ast(Y)$. Thus, the equation $m^\ast(\theta_\beta) = \pi_2^\ast(\theta_\beta) \circ \pi_1^\ast(\theta_\beta)$ type-checks.

The map $\pi_1^\ast(\theta_\beta)$ is the map in the double pullback
\[
\begin{tikzcd}
    \pi_1^\ast s^\ast (Y) \ar{r} \ar{d}[']{\pi_1^\ast(\theta_\beta)} \arrow[dr, phantom, "\lrcorner", very near start] & s^\ast(Y) \ar{d}{\theta_\beta} \\
    \pi_1^\ast t^\ast (Y) \ar{r} \ar{d} \arrow[dr, phantom, "\lrcorner", very near start] & t^\ast(Y) \ar{d} \\
    G_1 \times_{G_0} G_1 \ar[swap]{r}{\pi_1} & G_1,
\end{tikzcd}
\]
and therefore acts by $(y,f,g) \mapsto (\beta(y,f),f,g)$.  Similarly, $\pi_2^\ast(\theta_\beta) \colon \pi_2^\ast s^\ast(Y) \to \pi_2^\ast t^\ast(Y) $ acts by $(y,f,g) \mapsto (\beta(y,g),f,g)$ and $m^\ast (\theta_\beta) \colon m^\ast s^\ast(Y) \to m^\ast t^\ast(Y)$ acts by $(y,f,g) \mapsto (\beta(y,m(f,g)),f,g)$.  Thus, we observe that
\begin{align*}
    (\pi_2^\ast(\theta_\beta) \circ \pi_1^\ast(\theta_\beta)) (y,f,g) &= \pi_2^\ast(\theta_\beta)  (\beta(y,f),f,g) \\
    &= (\beta(\beta(y,f),g),f,g) \\
    &= (\beta(y,m(f,g)),f,g) \\
    &= m^\ast(\theta_\beta)(y,f,g).
\end{align*}
Hence, the pair $(Y,\theta_\beta)$ indeed constitutes descent datum.

An equivariant map $f \colon Y \to Y'$ between spaces with respective $G_1$-actions $\beta$ and $\beta'$ also constitutes a morphism of descent data $f \colon (Y,\theta_\beta) \to (Y',\theta_{\beta'})$.  The required commutativity condition, $t^\ast(f) \circ \theta_\beta = \theta_{\beta'} \circ s^\ast(f)$, is forced by universal property of $t^\ast(Y')$ in the following commutative diagram.
\[\begin{tikzcd}[sep=small]
	{s^\ast(Y)} && {s^\ast(Y')} \\
	\\
	Y && {Y'} \\
	& {t^\ast(Y)} && {t^\ast(Y')} \\
	{G_0} && {G_0} \\
	& {G_1} && {G_1}
	\arrow["{s^\ast(f)}", from=1-1, to=1-3]
	\arrow["\beta"', from=1-1, to=3-1]
	\arrow["{\beta'}"', from=1-3, to=3-3]
	\arrow["f"{pos=0.6}, from=3-1, to=3-3]
	\arrow[from=3-1, to=5-1]
	\arrow[from=3-3, to=5-3]
	\arrow[from=5-1, to=5-3, equal]
	\arrow["t"', from=6-2, to=5-1]
	\arrow["t"', from=6-4, to=5-3]
	\arrow[from=6-2, to=6-4, equal]
	\arrow[from=4-4, to=6-4]
	\arrow[from=4-4, to=3-3]
	\arrow[from=4-2, to=3-1]
	\arrow["{t^\ast(f)}"{pos=0.25}, from=4-2, to=4-4, crossing over]
	\arrow["{\theta_\beta}", from=1-1, to=4-2, crossing over]
	\arrow["{\theta_{\beta'}}", from=1-3, to=4-4]
        \arrow[from=4-2, to=6-2, crossing over]
\end{tikzcd}\]

For the other direction, suppose we are given a descent datum $(Y,\theta)$. We then obtain a compatible $G_1$-action $\beta_\theta \colon s^\ast(Y) \to Y$ by taking $\beta_\theta$ to be the composite
\[
\begin{tikzcd}
     Y \times_{G_0} G_1 \cong s^\ast(Y) \ar{r}{\theta} & t^\ast(Y) \ar{r} & Y.
\end{tikzcd}
\]
Checking that $\beta_\theta$ is a legitimate $G_1$-action or that a morphism of descent data $f \colon (Y,\theta) \to (Y',\theta')$ yields an equivariant map $f \colon (Y,\beta_\theta) \to (Y',\beta_{\theta'})$ essentially amounts to the reverse of what we have done above and so we omit the details. Finally, note that the two correspondences are mutual inverses.

\bibliographystyle{abbrv}
\bibliography{references}

\end{document}